\documentclass{amsart}
\usepackage{amsfonts}

\newtheorem{theorem}{Theorem}[section]
\newtheorem{lemma}[theorem]{Lemma}
\newtheorem{prop}[theorem]{Proposition}
\newtheorem{coro}[theorem]{Corollary}

\theoremstyle{definition}

\theoremstyle{remark}
\newtheorem{remark}[theorem]{Remark}

\numberwithin{equation}{section}

\begin{document}

\title[Mean curvature flow of higher codimension]{Mean curvature flow of higher codimension in Riemannian manifolds}

\author{Kefeng Liu}
\address{Center of Mathematical Sciences, Zhejiang University,
             Hangzhou,  310027, People¡¯s Republic of China; Department of Mathematics, UCLA, Box 951555, Los Angeles, CA, 90095-1555 }
\email{liu@cms.zju.edu.cn, liu@math.ucla.edu}

\author{Hongwei Xu}

\address{Center of Mathematical Sciences, Zhejiang University,
             Hangzhou,  310027, People¡¯s Republic of China}
\email{xuhw@cms.zju.edu.cn}

\author{Entao Zhao}
\address{Center of Mathematical Sciences, Zhejiang University,
             Hangzhou,  310027, People¡¯s Republic of China}
\email{zhaoet@cms.zju.edu.cn}

\thanks{Supported by the National Natural Science Foundation
of China, Grant No. 11071211; the Trans-Century Training Programme
Foundation for Talents by the Ministry of Education of China, and
the China Postdoctoral Science Foundation, Grant No. 20090461379.}


\date{}


\begin{abstract}We investigate the convergence of the mean curvature flow of
arbitrary codimension in Riemannian manifolds with bounded geometry.
We prove that if the initial submanifold satisfies a pinching
condition, then along the mean curvature flow the submanifold
contracts smoothly to a round point in finite time. As a consequence
we obtain a differentiable sphere theorem for submanifolds in a
Riemannian manifold.
\end{abstract}

\maketitle

\section{Introduction}

\label{intro}Let $F_0:\, M^{n}\rightarrow N^{n+d}$ be a smooth
immersion from an $n$-dimensional Riemannian manifold without
boundary to an $(n+d)$-dimensional Riemannian manifold. Consider a
one-parameter family of smooth immersions $F:\, M\times
[0,T)\rightarrow N$ satisfying
\begin{eqnarray}
\label{MCF}\left\{
\begin{array}{ll}
\frac{\partial}{\partial t}F(x,t)=H(x,t),\\
F(x,0)=F_0(x),
\end{array}\right.
\end{eqnarray}
where $H(x,t)$ is the mean curvature vector of $F_t(M)$ and
$F_t(x)=F(x,t)$. We call $F:\, M\times [0,T)\rightarrow N$ the mean
curvature flow with initial value $F_0:\, M\rightarrow N$.

The mean curvature flow was proposed by Mullins \cite{Mu} to
describe the formation of grain boundaries in annealing metals. In
\cite{B}, Brakke introduced the motion of a submanifold by its mean
curvature in arbitrary codimension and constructed a generalized
varifold solution for all time. For the classical solution of the
mean curvature flow, many works on hypersurfaces have been done.
Huisken \cite{H1} showed that if the initial hypersurface in the
Euclidean space is compact and uniformly convex, then the mean
curvature flow converges to a round point in a finite time. Later,
he generalized this convergence theorem to the mean curvature flow
of hypersurfaces in a Riemannian manifold in \cite{H2}. He also
studied in \cite{H3} the mean curvature flow of hypersurfaces
satisfying a pinching condition in a sphere.

For the mean curvature flow of submanifolds with higher
codimensions, fruitful results were obtained for submanifolds with
low dimension or admitting some special structures, see
\cite{Sm,SW,WaM1,WaM4,WaM2,WaM5,WaM6,WaM3} etc. for example.
Recently, Andrews and Baker \cite{Andrews-Baker} proved a
convergence theorem for the mean curvature flow of closed
submanifolds satisfying a suitable pinching condition in the
Euclidean space. In \cite{LXYZ1,LXZ}, the authors of the present
paper and Ye investigated the integral curvature pinching conditions
that assure the convergence of the mean curvature flow of
submanifolds in an Euclidean space or a sphere. More recently, Baker
\cite{Baker} and Liu-Xu-Ye-Zhao \cite{LXYZ2} generalized
Andrews-Baker's convergence theorem \cite{Andrews-Baker} for the
mean curvature flow of submanifolds in the Euclidean space to the
case of the mean curvature flow of arbitrary codimension in space
forms.

In this paper, we study the convergence of the mean curvature flow
of submanifolds in a general Riemannian manifold.  Let $F:\,
M\rightarrow N$ be a smooth submanifold. Suppose the sectional
curvature $K_N$, the first covariant derivative
$\bar{\nabla}\bar{R}$ of the Riemannian curvature tensor, and the
injectivity radius ${\rm inj}(N)$ of the ambient space $N$ satisfy
\begin{eqnarray}
\label{K_N} -K_1\leq  K_N \leq K_2,\\
\label{1st-deri-N} |\bar{\nabla}\bar{R}|\leq
L,\ \ \ \ \\
\label{inj-radius-N} {\rm inj}(N)\geq i_N,\ \ \
\end{eqnarray}
for nonnegative constants $K_1$, $K_2$, $L$ and positive constant
$i_N$. Our main result is the following:

\begin{theorem}\label{main-thm}
Let $F:\,M^{n} \rightarrow N^{n+d}$ be an $n$-dimensional smooth
closed and connected submanifold in an $(n+d)$-dimensional smooth
complete Riemannian manifold satisfying
(\ref{K_N})-(\ref{inj-radius-N}). There is an explicitly computable
nonnegative constant $b_0$ depending on $n$, $d$, $K_1$, $K_2$ and
$L$ such that if $F$ satisfies
\begin{eqnarray}
\label{pinch-cond} |A|^2<\begin{cases}
            \frac{4}{3n}|H|^2-b_0, \ &n = 2, 3, \\
            \frac{1}{n-1}|H|^2-b_0, \ &n \geq 4,
        \end{cases}
\end{eqnarray}
then the mean curvature flow with $F$ as initial value contracts to
a round point in finite time.
\end{theorem}

Theorem \ref{main-thm} can be considered as an extension of the
convergence result of Huisken in \cite{H2} to higher codimension
case under a curvature pinching condition rather than the convexity
of the initial hypersurface. On the other hand, if
$N=\mathbb{R}^{n+d}$, then $b_0=0$. From Proposition 7 of
\cite{Andrews-Baker} we see that, under their initial curvature
pinching condition (\ref{pinch-cond}) is satisfied after a short
time interval. Hence our Theorem \ref{main-thm} may also be
considered as a generalization of the convergence theorem in
\cite{Andrews-Baker}.

By the Nash imbedding theorem, every compact Riemannian manifold can
be isometrically embedded into an Euclidean space or a higher
dimensional Riemannian manifold as a submanifold, in general of
higher codimension. By using mean curvature flow techniques
developed in this paper we can study certain important problems in
Riemannian geometry which will be the content of our forthcoming
works.

As a consequence of Theorem \ref{main-thm}, we obtain the following
differentiable sphere theorem for submanifolds in a Riemannian
manifold.

\begin{coro}Under the assumption of Theorem \ref{main-thm}, $M$ is
diffeomorphic to the standard unit $n$-sphere $\mathbb{S}^n$.
\end{coro}

\begin{remark}
In \cite{Gu-Xu,XG,Xu-Zhao}, some differentiable sphere theorems for
simply connected submanifolds in certain Riemannian manifolds were
obtained by using convergence results for the Ricci flow.
\end{remark}

The paper is organized as follows. In Section 2, we introduce some
basic equations in submanifold theory, and recall some evolution
equations along the mean curvature flow. In Section 3, we show that
the pinching condition (\ref{pinch-cond}) for a suitable $b_0$ is
preserved along the mean curvature flow. A pinching estimate for the
tracefree second fundamental form is obtained in Section 4, which
implies that the submanifold becomes spherical as $t$ tends to the
maximal existence time. We also show that, under the initial
pinching condition, the maximal existence time is finite. We give an
estimate of the gradient of the mean curvature in Section 5, which
is used to compare the mean curvature at different points. In
Section 6, we show that the submanifold shrinks to a single point in
finite time. After the dilation of the ambient space and a
reparameterization of time, the ambient space will converges to the
Euclidian space and the submanifold will converges to a totally
umbilical sphere with the same volume as the initial submanifold.

\section{Preliminaries}

\label{Pre} Let $F:\, M^{n}\rightarrow N^{n+d}$ be a smooth immersion
from an $n$-dimensional Riemannian manifold $M^n$ without boundary
to an $(n+d)$-dimensional Riemannian manifold $N^{n+d}$. We shall
make use of the following convention on the range of indices:
$$1\leq i,j,k,\cdots \leq n,\ \ 1\leq A,B,C, \cdots \leq n+d,\ \
and\ \ n+1\leq\alpha,\beta,\gamma, \cdots \leq n+d.$$

Choose a local orthonormal frame field $\{e_A\}$ in $N$ such that
$e_i$'s are tangent to $M$. Let $\{\omega_A\}$ be the dual frame
field of  $\{e_A\}$. The metric $g$ and the volume form $d\mu$ of
$M$ are $g=\sum_i \omega_i\otimes\omega_i$ and $d\mu
=\omega_1\wedge\cdots\wedge\omega_n$.

 For any $x\in M$, denoted by $N_xM$
the normal space of $M$ at point $x$, which is  the orthogonal
complement of $T_xM$ in $F^{\ast}T_{F(x)}N$. Here we identify $T_xM$
with its image under the map $F_\ast$. Denote by $\bar{\nabla}$ the
Levi-Civita connection on $N$. The Riemannian curvature tensor
$\bar{R}$  of $N$ is defined by
\begin{equation*}\bar{R}(U,V)W=-\bar{\nabla}_U\bar{\nabla}_VW+\bar{\nabla}_V\bar{\nabla}_UW+\bar{\nabla}_{[U,V]}W
\end{equation*}
for vector fields $U,V$ and $ W$ tangent to $N$. The induced
connection $\nabla$ on $M$ is defined by
\begin{equation*}\nabla_XY=(\bar{\nabla}_XY)^{\top},
\end{equation*} for $X,Y$ tangent to $M$, where $(\ )^\top$ denotes
tangential component. Let $R$ be the Riemannian curvature tensor of
$M$.

Given a normal vector field $\xi$ along $M$, the induced connection
$\nabla^\bot$ on the normal bundle is defined by
\begin{equation*}\nabla^\bot
_X\xi=(\bar{\nabla}_X\xi)^{\bot},\end{equation*} where $(\ )^{\bot}$
denotes the normal component. Let $R^\bot$ denote the normal
curvature tensor.

 The second fundamental form is defined to be
$$A(X,Y)=(\bar{\nabla}_XY)^\bot$$ as a section of the tensor bundle
$T^\ast M\otimes T^\ast M\otimes NM$, where $T^\ast M$ and $NM$ are
the cotangential bundle and  the normal bundle on $M$. The mean
curvature vector $H$ is the trace of the second fundamental form
defined by $H={\rm tr}_{g}A$.

The first covariant derivative of $A$ is defined as
\begin{equation*}
(\widetilde{\nabla}_XA)(Y,Z)=\nabla^\bot_XA(Y,Z)-A(\nabla_XY,Z)-A(Y,\nabla_XZ),
\end{equation*}
where $\widetilde{\nabla}$ is the connection on $T^\ast M\otimes
T^\ast M\otimes NM$. Similarly, we can define the second covariant
derivative of $A$.

Under the local orthonormal frame field, the components of second
fundamental form and its first and second covariant derivatives of
$A$ are defined by
\begin{equation*}\begin{split}
h^\alpha_{ij}=&\langle
A(e_i,e_j), e_\alpha\rangle,\\
\nabla _kh^\alpha_{ij}=&\langle (\widetilde{\nabla}
_{e_k}A)(e_i,e_j),
e_\alpha\rangle,\\
\nabla_l\nabla_kh^\alpha_{ij}=&\langle (\widetilde{\nabla}
_{e_l}\widetilde{\nabla}_{e_k}A)(e_i,e_j), e_\alpha\rangle.
\end{split}
\end{equation*}
The Laplacian of $A$ is defined by $\Delta h^\alpha_{ij}=\sum_k
\nabla_k\nabla_kh^{\alpha}_{ij}$.

We define the tracefree second fundamental form by
$\mathring{A}=A-\frac{1}{n}g\otimes H$, whose components are
$\mathring{h}^\alpha_{ij}=h^\alpha_{ij}-\frac{1}{n}H^\alpha
\delta_{ij}$, where $H^\alpha=\sum_kh^\alpha_{kk}$. Obviously, we
have $\sum_i\mathring{h}^\alpha_{ii}=0$.

Let \begin{equation*}
\begin{split}
R_{ijkl}&=g(R(e_i,e_j)e_k,e_l),\\
\bar{R}_{ABCD}&=\langle\bar{R}(e_A,e_B)e_C,e_D\rangle,\\
R^\bot_{ij\alpha\beta}&=\langle{R^\bot}(e_i,e_j)e_\alpha,e_\beta\rangle.
\end{split}
\end{equation*}
Then we have the following Gauss, Codazzi and Ricci equations.
\begin{equation*}
\begin{split}
R_{ijkl}&=\bar{R}_{ijkl}+\sum_\alpha\Big( h^\alpha_{ik}h^\alpha_{jl} -h^\alpha_{il}h^\alpha_{jk}\Big),\\
\nabla_kh^\alpha_{ij}-\nabla_j h^\alpha_{ik}&=-\bar{R}_{\alpha ijk},\\
R^\bot_{ij\alpha\beta}&=\bar{R}_{ij\alpha\beta}+\sum_k\Big(h^\alpha_{ik}h^\beta_{jk}-h^\alpha_{jk}h^\beta_{ik}\Big).
\end{split}
\end{equation*}

It is standard to show the short-time existence of the mean
curvature flow (\ref{MCF}) with closed initial value. Since the mean
curvature flow is a (degenerate) quasilinear parabolic evolution
equation, one can obtain the short-time existence by using the
Nash-Moser implicit function theorem as in \cite{Ha}. One can
also use the De Turck trick to modify the mean curvature
flow equation to a strongly parabolic equation, and the short-time
existence follows from the standard parabolic theory.

Let $F:\, M^n\times [0,T)\rightarrow N^{n+d}$ be a mean curvature flow
solution. We have the following evolution equations.

\begin{lemma}\label{lem-evolution}Along the mean curvature flow we have
\begin{equation}\label{evo-volform}
\frac{\partial}{\partial t}d\mu_t=-|H|^2d\mu_t,
\end{equation}

\begin{equation}\label{evu-|A|^2}
\begin{split}
\frac{\partial}{\partial t}|A|^2=&\Delta |A|^2-2|\nabla
A|^2+2\sum_{\alpha,\beta}\Big(\sum_{i,j}h^\alpha_{ij}h^{\beta}_{ij}\Big)^2
+2\sum_{i,j,\alpha,\beta}\Big[\sum_{p}\Big(h_{ip}^\alpha
h_{jp}^\beta-h_{jp}^\alpha h_{ip}^\beta\Big)\Big]^2\\
&+4\sum_{i,j,p,q}\bar{R}_{ipjq}\Big(\sum_{\alpha} h_{pq}^\alpha
h_{ij}^\alpha\Big)-4\sum_{j,k,p}\bar{R}_{kjkp}\Big(\sum_{i,\alpha}h^\alpha_{pi}h^\alpha_{ij}\Big)\\
&+2\sum_{k,\alpha,\beta}\bar{R}_{k\alpha
k\beta}\Big(\sum_{i,j}h_{ij}^\alpha
h_{ij}^\beta\Big)-8\sum_{j,p,\alpha,\beta}\bar{R}_{jp\alpha
\beta}\Big(\sum_{i}h_{ip}^\alpha h_{ij}^\beta\Big)\\
&+2\sum_{i,j,k,\beta}\bar{\nabla}_k
\bar{R}_{kij\beta}h_{ij}^\beta-2\sum_{i,j,k,\beta}\bar{\nabla}_i
\bar{R}_{jkk\beta}h_{ij}^\beta,
\end{split}\end{equation}
\begin{equation}\label{evu-|H|^2}
\begin{split}
\frac{\partial}{\partial t}|H|^2=&\Delta |H|^2-2|\nabla
H|^2+2\sum_{i,j}\Big(\sum_{\alpha}H^{\alpha}h^\alpha_{ij}\Big)^2+2\sum_{k,\alpha,\beta}\bar{R}_{k\alpha
k\beta}H^\alpha H^\beta.
\end{split}\end{equation}
\end{lemma}

Throughout this paper, we assume that the submanifold is connected,
and the ambient space $N$ satisfies (\ref{K_N})-(\ref{inj-radius-N})
for nonnegative constants $K_1$, $K_2$, $L$ and positive constant
$i_N$. By Berger's inequality (see \cite{Goldberg} for a proof), we
see that the $|\bar{R}_{ACBC}|\leq \frac{1}{2}(K_1+K_2)$ for $A\neq
B$ and $|\bar{R}_{ABCD}|\leq \frac{2}{3}(K_1+K_2)$ for all distinct
indices $A, B, C, D$.

\section{A preserved curvature pinching condition}

In this section, we prove that the  pinching condition
(\ref{pinch-cond}) for a suitable $b_0>0$ is preserved along the
mean curvature. But first we prove the following lemma.

\begin{lemma}\label{lem-gradient-ineq}For any $\eta>0$ we have the following inequalities.
\begin{equation}
\label{gradient-ineq-1} |\nabla
A|^2\geq\bigg(\frac{3}{n+2}-\eta\bigg)|\nabla H|^2
-\frac{2}{n+2}\bigg(\frac{2}{n+2}\eta^{-1}-\frac{n}{n-1} \bigg)
|w|^2,
\end{equation}
\begin{equation}
\begin{split}
 \label{gradient-ineq-2}|\nabla {A}|^2-\frac{1}{n}|\nabla
H|^2\geq&
\frac{n-1}{2n+1}|\nabla A|^2-\frac{2n}{(n-1)(2n+1)}|w|^2 \\
\geq&\frac{n-1}{2n+1}|\nabla A|^2-C(n,d)(K_1+K_2)^2.
\end{split}
\end{equation}
Here  $w=\sum_{i,j,\alpha} \bar{R}_{\alpha jij}
e_i\otimes\omega_\alpha$ and $C(n,d)=\frac{n^4d}{2(n-1)(2n+1)}$.
\end{lemma}

\begin{proof}Inequality (\ref{gradient-ineq-2}) follows from (\ref{gradient-ineq-1}) with
$\eta=\frac{n-1}{n(n+2)}$. To prove (\ref{gradient-ineq-1}), we set
\begin{equation*}
\begin{split}
E_{ijk}=&\frac{1}{n+2}(\nabla_i Hg_{jk}+\nabla_jH g_{ik}+\nabla_k H
g_{ij})\\
&-\frac{2}{(n+2)(n-1)}w_ig_{jk}+\frac{n}{(n+2)(n-1)}( w_j
g_{ik}+w_kg_{ij} ).
\end{split}
\end{equation*}
Let $F_{ijk}=\nabla_i h_{jk}-E_{ijk}.$  By the Codazzi equation we
have $\langle E_{ijk}, F_{ijk}\rangle=0$. Hence $|\nabla A|^2\geq
|E|^2$. By a direct computation, we have
\begin{equation*}
|E|^2=\frac{3}{n+2}|\nabla
H|^2+\frac{2n}{(n+2)(n-1)}|w|^2+\frac{4}{n+2}\langle \nabla
H,w\rangle.
\end{equation*}
Then (\ref{gradient-ineq-1}) follows from Schwartz's inequality,
Young's inequality and Berger's inequality.
\end{proof}

\begin{theorem}\label{pres-pinching}There is a positive constant $b_1$ depending on
 $n$, $d$, $K_1$, $K_2$, $L$  and $a$ such that if  $|A|^2\leq a|H|^2-b$
 holds for some constant $a\leq \frac{4}{3n}$ and $b> b_1$ at $t=0$, then it
 remains true for $t>0$.
\end{theorem}

\begin{proof}
Set $Q=|A|^2-a|H|^2+b$, where $a\leq \frac{4}{3n}$, $b> b_1$, and
$b_1$ is a positive constant to be determined. We will compute the
evolution of $Q$ along the mean curvature flow, and show that if
$Q=0$ at a point in the space-time, then $(\frac{\partial}{\partial
t}-\Delta) Q$ is negative at this point. By the maximum principle,
the theorem follows.

By Lemma \ref{lem-evolution}, we have
\begin{equation}\label{evo-Q}
\begin{split}
\frac{\partial}{\partial t}Q=&\Delta Q-2(|\nabla A|^2-a|\nabla
H|^2)+2R_1-2aR_2+{\rm P}_{a},
\end{split}\end{equation}
where
\begin{equation*}
R_1=\sum_{\alpha,\beta}\Big(\sum_{i,j}h^\alpha_{ij}h^{\beta}_{ij}\Big)^2+\sum_{i,j,\alpha,\beta}\Big[\sum_{p}\Big(h_{ip}^\alpha
h_{jp}^\beta-h_{jp}^\alpha h_{ip}^\beta\Big)\Big]^2,
\end{equation*}
\begin{equation*}
R_2=\sum_{i,j}\Big(\sum_{\alpha}H^{\alpha}h^\alpha_{ij}\Big)^2,
\end{equation*}
and  ${\rm P}_a=I+II+III+IV$ with
\begin{equation*}
I=4\sum_{i,j,p,q}\bar{R}_{ipjq}\Big(\sum_{\alpha} h_{pq}^\alpha
h_{ij}^\alpha\Big)-4\sum_{j,k,p}\bar{R}_{kjkp}\Big(\sum_{i,\alpha}h^\alpha_{pi}h^\alpha_{ij}\Big),
\end{equation*}
\begin{equation*}
II=2\sum_{k,\alpha,\beta}\bar{R}_{k\alpha
k\beta}\Big(\sum_{i,j}h_{ij}^\alpha
h_{ij}^\beta\Big)-2a\sum_{k,\alpha,\beta}\bar{R}_{k\alpha
k\beta}H^\alpha H^\beta,\end{equation*}
\begin{equation*}
III=-8\sum_{j,p,\alpha,\beta}\bar{R}_{jp\alpha
\beta}\Big(\sum_{i}h_{ip}^\alpha h_{ij}^\beta\Big),\end{equation*}
\begin{equation*}
IV=2\sum_{i,j,k,\beta}\bar{\nabla}_k
\bar{R}_{kij\beta}h_{ij}^\beta-2\sum_{i,j,k,\beta}\bar{\nabla}_i
\bar{R}_{jkk\beta}h_{ij}^\beta.
\end{equation*}

At the point  where $Q=0$, the mean curvature vector is not zero.
Hence we choose $e_{n+1}=\frac{H}{|H|}$. The second fundamental form
can be written as $A=\sum_{\alpha}h^\alpha e_{\alpha}$, where
$h^\alpha$, $n+1\leq \alpha \leq n+d$, are symmetric 2-tensors. By
the choice of $e_{n+1}$, we see that $H^{n+1}={\rm tr} h^{n+1}=|H|$
and $H^\alpha={\rm tr} h^\alpha=0$ for $\alpha\geq n+2$. The
tracefree second fundamental form may be  rewritten as
$\mathring{A}=\sum_{\alpha}\mathring{h}^\alpha e_{\alpha}$, where
$\mathring{h}^{n+1}={h}^{n+1}-\frac{|H|}{n}{\rm Id}$ and
$\mathring{h}^\alpha={h}^\alpha$ for $\alpha\geq n+2$.  We set
$$|A|_H^2=|h^{n+1}|^2,\ \ |A|_I^2=\sum_{\alpha\geq
n+2}|h^{\alpha}|^2=|A|^2-|A|_H^2,$$
$$|\mathring{A}|_H^2=|\mathring{h}^{n+1}|^2,\ \ |\mathring{A}|_I^2=\sum_{\alpha\geq
n+2}|\mathring{h}^{\alpha}|^2=|\mathring{A}|^2-|\mathring{A}|_H^2.$$
Notice that $|A|_H^2=|\mathring{A}|_H^2+\frac{|H|^2}{n}$ and
$|A|_I^2=|\mathring{A}|_I^2$.

Since $Q=0$ at this point, we have
$|H|^2=\frac{|\mathring{A}|^2+b}{a-\frac{1}{n}}$. By the computation
in \cite{Andrews-Baker} we have
\begin{equation}\label{second-line}
\begin{split}
2R_1-2aR_2\leq&\bigg(6-\frac{2}{n(a-\frac{1}{n})}\bigg)|\mathring{A}|_H^2|\mathring{A}|_I^2
+\bigg(3-\frac{2}{n(a-\frac{1}{n})}\bigg)|\mathring{A}|_I^4\\
&-\frac{2nab}{n(a-\frac{1}{n})}|\mathring{A}|_H^2-\frac{4b}{n(a-\frac{1}{n})}|\mathring{A}|_I^2-\frac{2b^2}{n(a-\frac{1}{n})}.
\end{split}\end{equation}

To estimate $I$, we fix   $\alpha$ and choose $e_i$'s such that
$h^{\alpha}_{ij}=\lambda^\alpha_i \delta_{ij}$. Then
\begin{equation*}\begin{split}
&4\sum_{i,j,p,q}\bar{R}_{ipjq}h_{pq}^\alpha
h_{ij}^\alpha-4\sum_{j,k,p}\bar{R}_{kjkp}\Big(\sum_{i}h^\alpha_{pi}h^\alpha_{ij}\Big)\\
&=4\sum_{i,p}\bar{R}_{ipip}\Big(\lambda^\alpha_i\lambda^\alpha_p-(\lambda^\alpha_i)^2\Big)\\
&=-2\sum_{i,p}\bar{R}_{ipip}\Big(\lambda^\alpha_i-\lambda^\alpha_p\Big)^2\\
&\leq 4nK_1|\mathring{h}^\alpha|^2.
\end{split}
\end{equation*}
Hence we get
\begin{equation}\label{I-estimate}
I\leq 4nK_1(|\mathring{A}|_H^2+|\mathring{A}|_I^2).
\end{equation}

By the choice of $e_{n+1}$, we have
\begin{equation*}
II=II_1+II_2+II_3,
\end{equation*}
where
\begin{equation*}
II_1=2\sum_{i,j,k}\bar{R}_{k n+1
kn+1}(h_{ij}^{n+1})^2-2a\sum_{k}\bar{R}_{k n+1 k n+1} (H^{n+1})^2,
\end{equation*}
\begin{equation*}
II_2=4\sum_{k,\alpha\geq n+2}\bar{R}_{k\alpha
kn+1}\Big(\sum_{i,j}h_{ij}^\alpha
h_{ij}^{n+1}\Big)-4a\sum_{k,\alpha\geq n+2}\bar{R}_{k\alpha kn+1}
H^{n+1} H^{\alpha},
\end{equation*}
\begin{equation*}
II_3=2\sum_{k,\alpha,\beta\geq n+2}\bar{R}_{k\alpha
k\beta}\Big(\sum_{i,j}h_{ij}^\alpha
h_{ij}^\beta\Big)-2a\sum_{k,\alpha,\beta\geq n+2}\bar{R}_{k\alpha
k\beta}H^\alpha H^\beta.
\end{equation*}

Since $|H|^2=\frac{|\mathring{A}|^2+b}{a-\frac{1}{n}}$ at that
point, we have
\begin{equation*}\begin{split}
II_1\leq& 2nK_2|A|_H^2+2na K_1|H|^2\\
=&2nK_2|\mathring{A}|_H^2+2(naK_1+K_2)\cdot
\frac{|\mathring{A}|^2+b}{a-\frac{1}{n}}\\
=& \bigg(2nK_2+
\frac{2(naK_1+K_2)}{a-\frac{1}{n}}\bigg)|\mathring{A}|_H^2 +
\frac{2(naK_1+K_2)}{a-\frac{1}{n}}|\mathring{A}|_I^2+\frac{2(naK_1+K_2)b}{a-\frac{1}{n}}.
\end{split}\end{equation*}

Since $H^\alpha=0$ for $\alpha\geq n+2$, we have the following
estimates for $II_2$ and $II_3$.
\begin{equation*}\begin{split}
II_2=&4\sum_{k,\alpha\geq n+2}\bar{R}_{k\alpha
kn+1}\Big(\sum_{i,j}h_{ij}^\alpha h_{ij}^{n+1}\Big)\\
=&4\sum_{k,\alpha\geq n+2}\bar{R}_{k\alpha
kn+1}\Big(\sum_{i,j}\mathring{h}_{ij}^\alpha \mathring{h}_{ij}^{n+1}\Big)\\
\leq&(K_1+K_2)\sum_{k,\alpha\geq n+2}\Big(
\frac{1}{\varrho}\sum_{i,j}(\mathring{h}_{ij}^\alpha)^2
+\varrho\sum_{i,j}(\mathring{h}_{ij}^{n+1})^2\Big)\\
=&{\varrho}n(d-1)(K_1+K_2)|\mathring{A}|_H^2+\frac{n}{\varrho}(K_1+K_2)|\mathring{A}|_I^2,
\end{split}\end{equation*}
for any positive constant $\varrho$.
\begin{equation*}\begin{split}
II_3=&2\sum_{k,\alpha,\beta\geq n+2}\bar{R}_{k\alpha
k\beta}\Big(\sum_{i,j}h_{ij}^\alpha h_{ij}^\beta\Big)\\
=&2\sum_{k,\alpha\geq n+2}\bar{R}_{k\alpha
k\alpha}\Big(\sum_{i,j}h_{ij}^\alpha)^2\Big)+2\sum_{k,\alpha,\beta\geq
n+2,\alpha\neq\beta}\bar{R}_{k\alpha
k\beta}\Big(\sum_{i,j}h_{ij}^\alpha h_{ij}^\beta\Big)\\
\leq&2nK_2|\mathring{A}|_I^2+2\sum_{k,\alpha,\beta\geq
n+2,\alpha\neq\beta}\bar{R}_{k\alpha
k\beta}\Big(\sum_{i,j}h_{ij}^\alpha h_{ij}^\beta\Big)\\
\leq&2nK_2|\mathring{A}|_I^2+\sum_{k,\alpha,\beta\geq
n+2,\alpha\neq\beta}|\bar{R}_{k\alpha
k\beta}|\Big(\sum_{i,j}(h_{ij}^\alpha)^2+ (h_{ij}^\beta)^2\Big)\\
\leq&2nK_2|\mathring{A}|_I^2+(K_1+K_2)\sum_{i,j,k,\alpha,\beta\geq
n+2,\alpha\neq\beta}(h_{ij}^\alpha)^2\\
=&2nK_2|\mathring{A}|_I^2+n(d-2)(K_1+K_2)|\mathring{A}|_H^2.
\end{split}\end{equation*}

Hence we get the following estimate for $II$.
\begin{equation}\label{II-estimate}\begin{split}
II\leq &\bigg(  2nK_2+
\frac{2(naK_1+K_2)}{a-\frac{1}{n}}+[{\varrho}n(d-1)+n(d-2)](K_1+K_2)
\bigg) |\mathring{A}|_H^2\\
&+\bigg( \frac{2(naK_1+K_2)}{a-\frac{1}{n}}+
\frac{n}{\varrho}(K_1+K_2) +2nK_2  \bigg) |\mathring{A}|_I^2\\
&+\frac{2(naK_1+K_2)b}{a-\frac{1}{n}}.
\end{split}\end{equation}

For $III$, we have
\begin{equation*}
III=III_1+III_2,
\end{equation*}
where
\begin{equation*}
III_1=-16\sum_{j,p,\alpha\geq n+2}\bar{R}_{jp\alpha
n+1}\Big(\sum_{i}h_{ip}^\alpha h_{ij}^{n+1}\Big), \end{equation*}
\begin{equation*}
III_2=-8\sum_{j,p,\alpha,\beta\geq
n+2,\alpha\neq\beta}\bar{R}_{jp\alpha\beta}\Big(\sum_{i}h_{ip}^\alpha
h_{ij}^\beta\Big). \end{equation*}

We have the following estimates for arbitrary positive constant
$\rho$.
\begin{equation*}\begin{split}
III_1=&-16\sum_{j,p,\alpha\geq n+2}\bar{R}_{jp\alpha
n+1}\Big(\sum_{i}\mathring{h}_{ip}^\alpha \Big(
\mathring{h}_{ij}^{n+1}+\frac{|H|}{n}\delta_{ij}\Big)\Big)\\
=&-16\sum_{j\neq p,\alpha\geq n+2}\bar{R}_{jp\alpha
n+1}\Big(\sum_{i}\mathring{h}_{ip}^\alpha
\mathring{h}_{ij}^{n+1}\Big)\\
\leq&\frac{16}{3}(K_1+K_2)\sum_{j\neq p,i,\alpha\geq
n+2}\Big(\frac{1}{\rho}(\mathring{h}_{ip}^\alpha)^2+\rho
(\mathring{h}_{ij}^{n+1})^2\Big)\\
=&\frac{16}{3}\rho(n-1)(d-1)(K_1+K_2)|\mathring{A}|_H^2+\frac{16}{3\rho}(n-1)(K_1+K_2)|\mathring{A}|_I^2.
\end{split}\end{equation*}
Here for the second equality, we use the fact that
$\sum_{j,p}\bar{R}_{jp\alpha n+1}\mathring{h}^\alpha_{jp}=0$ since
$\bar{R}_{jp\alpha n+1}$ is anti-symmetric for $j,p$ and
$\mathring{h}^\alpha_{jp}$ is symmetric for $j,p$.

For any fixed $\beta\geq n+2$, we choose $e_i$'s such that
$\mathring{h}^\beta_{ij}=\mathring{\lambda}^\beta_{i}\delta_{ij}$.
Then
\begin{equation*}
\begin{split}
III_2=&-8 \sum_{\beta\geq n+2} \sum_{j\neq p,\alpha\geq
n+2,\alpha\neq\beta}\bar{R}_{jp\alpha\beta} \mathring{h}_{jp}^\alpha
\mathring{\lambda}_{j}^\beta\\
\leq&\frac{8}{3}(K_1+K_2)\sum_{\beta\geq
n+2}\bigg((n-1)^{\frac{1}{2}} \sum_{j\neq p,\alpha\geq
n+2,\alpha\neq\beta}
(\mathring{h}_{jp}^\alpha)^2+\frac{1}{(n-1)^{\frac{1}{2}}}
\sum_{j\neq p,\alpha\geq
n+2,\alpha\neq\beta}(\mathring{\lambda}_{j}^\beta)^2 \bigg)\\
\leq&\frac{8}{3}(K_1+K_2)\bigg(
(n-1)^{\frac{1}{2}}(d-2)|\mathring{A}|_I^2 +
\sum_{\beta\geq n+2} (n-1)^{\frac{1}{2}}(d-2)|\mathring{h}^\beta|^2 \bigg)\\
=& \frac{8}{3}(n-1)^{\frac{1}{2}}(d-2)(K_1+K_2)|\mathring{A}|_I^2.
\end{split}
\end{equation*}

Hence we have
\begin{equation}\label{III-estimate}
\begin{split}
III\leq&\frac{16}{3}\rho(n-1)(d-1)(K_1+K_2)|\mathring{A}|_H^2\\
&+ \Big(\frac{16}{3\rho}(n-1)+
\frac{8}{3}(n-1)^{\frac{1}{2}}(d-2)\Big)(K_1+K_2)|\mathring{A}|_I^2.
\end{split}
\end{equation}

For $IV$, we choose $e_i$'s such that $h^{n+1}_{ij}=\lambda_i
\delta_{ij} $. If $K_1+K_2\neq0$, we have
\begin{equation}\label{IV-estimate}
\begin{split}
IV=&2\sum_{i,k}\bar{\nabla}_k
\bar{R}_{kiin+1}(\lambda_i-\lambda_k)-2\sum_{i,j,k,\beta\geq
n+2}(\bar{\nabla}_k \bar{R}_{kij\beta}-\bar{\nabla}_i
\bar{R}_{jkk\beta}) \mathring{h}_{ij}^\beta\\
\leq&\sum_{i,k}\Big( \frac{1}{\theta} (\bar{\nabla}_k
\bar{R}_{kiin+1})^2 +\theta (\lambda_i-\lambda_k)^2\Big)\\
&+\sum_{i,j,k,\beta\geq n+2}\Big(
\frac{2}{\vartheta}[(\bar{\nabla}_k
\bar{R}_{kij\beta})^2+(\bar{\nabla}_i \bar{R}_{jkk\beta})^2]+
\vartheta (\mathring{h}_{ij}^\beta)^2\Big)\\
\leq&\frac{L^2}{\theta}+\theta|\mathring{A}|_H^2+\frac{4L^2}{\vartheta}+n\vartheta|\mathring{A}|_I^2,
\end{split}\end{equation}
for positive constants $\theta,\vartheta$. If $K_1+K_2=0$, then
$L=0$, and we may choose $\theta,\vartheta=0$.

Combining (\ref{I-estimate})-(\ref{IV-estimate}), we have
\begin{equation}\label{14}
\begin{split} \Big(\frac{\partial}{\partial t}-\Delta\Big) Q\leq &\bigg(6-\frac{2}{n(a-\frac{1}{n})}\bigg)|\mathring{A}|_H^2|\mathring{A}|_I^2
+\bigg(3-\frac{2}{n(a-\frac{1}{n})}\bigg)|\mathring{A}|_I^4\\
&-\frac{2nab}{n(a-\frac{1}{n})}|\mathring{A}|_H^2-\frac{4b}{n(a-\frac{1}{n})}|\mathring{A}|_I^2-\frac{2b^2}{n(a-\frac{1}{n})}\\
&+C_1|\mathring{A}|_H^2+C_2|\mathring{A}|_I^2+C_3b+C_4.
\end{split}\end{equation}

Here
\begin{equation*}
\begin{split}
C_1=&4nK_1+  2nK_2+
\frac{2(naK_1+K_2)}{a-\frac{1}{n}}\\
&+\Big[{\varrho}n(d-1)+n(d-2)+\frac{16}{3}\rho(n-1)(d-1)\Big](K_1+K_2)+\theta,\\
C_2=&4nK_1+2nK_2+ \frac{2(naK_1+K_2)}{a-\frac{1}{n}}\\
&+ \Big(\frac{n}{\varrho}+\frac{16}{3\rho}(n-1)+
\frac{8}{3}(n-1)^{\frac{1}{2}}(d-2)\Big)(K_1+K_2)+n\vartheta,\\
C_3=&\frac{2(naK_1+K_2)}{a-\frac{1}{n}},\\
C_4=&\frac{L^2}{\theta}+\frac{4L^2}{\vartheta} \ \textrm{for}\
K_1+K_2\neq0\ \textrm{ and}\ C_4=0\  \textrm{for}\  K_1+K_2=0.
\end{split}\end{equation*}

If $K_1+K_2\neq0$, set
$b_1=\max\Big\{\frac{C_1}{2a}(a-\frac{1}{n}),\frac{C_2}{4}n(a-\frac{1}{n}),
\frac{1}{4}n(a-\frac{1}{n})\Big(C_3+\sqrt{C_3^2+\frac{8C_4}{n(a-\frac{1}{n})}}\Big)
\Big\}$ with   $\varrho=\rho=\theta=\vartheta=1$. If $K_1+K_2=0$,
set $b_1=0$. So if $b>b_1$,  we have $\Big(\frac{\partial}{\partial
t}-\Delta\Big) Q<0$. Then by the maximum principle, $|A|^2\leq
a|H|^2-b$ is preserved along the mean curvature flow.
\end{proof}

\begin{remark}
When $K_1+K_2\neq0$, we may get a better $b_1$ by choosing suitable
positive constants $\varrho,\ \rho,\ \theta$ and $\vartheta$.
\end{remark}

Now we pick the constant $b_0$ in (\ref{pinch-cond}) such that
$b_0\geq b_1$. Since the submanifold is compact, if
(\ref{pinch-cond}) is satisfied, then there are positive constants
$a_\varepsilon<a$ and $b_\varepsilon>b_1$, where $a$ denotes the
cofficient of $|H|^2$, such that $|A|^2\leq
a_\varepsilon|H|^2-b_\varepsilon$ holds at $t=0$, and it is
preserved along the mean curvature flow by Theorem
\ref{pres-pinching}. Hence in the remained part of the paper, we
always assume that $a_\varepsilon<a$, $b_\varepsilon>b_1$ and omit
the index $\varepsilon$.

\section{A pinching estimate for the tracefree second fundamental form}
In this section, we assume that at the initial time the submanifold
satisfies the pinching condition $|A|^2\leq a|H|^2-b$ for positive
constants $a,\ b$ such that  $a<\frac{4}{3n}$ when $n=2, 3$  and
 $a<\frac{1}{n-1}$ when $n\geq 4$, and  $b> b_1$, where $b_1$ is as
in Theorem \ref{pres-pinching}. From the last paragraph of Section 3
we see that the positive constants $a,\ b$ do exist under the
condition (\ref{pinch-cond}) and  the pinching condition is
preserved along the mean curvature. We prove a pinching estimate for
the tracefree second fundamental form, which guarantees that $M_t$
becomes totally umbilical along the mean curvature flow.
\begin{theorem}\label{thm-pinching-|A0|^2}
There are  constants $C_0<\infty$ and $\delta>0$ depending only on
$M_0$ such that along the mean curvature flow there holds
\begin{equation*}
|\mathring{A}|^2\leq C_0 |H|^{2-\delta}.
\end{equation*}
\end{theorem}

To prove Theorem \ref{thm-pinching-|A0|^2}, we define a function
$f_\sigma=\frac{|\mathring{A}|^2}{|H|^{2(1-\sigma)}}$ and wish to
find an upper bound of $f_\sigma$ for sufficiently small $\sigma$.
We first derive the evolution equation of $f_\sigma$.
\begin{prop}There is a constant $C$ depending only on $n$, $d$, $K_1$,
$K_2$ and $L$ such that along the mean curvature flow the following
evolution inequality holds.
\begin{equation}
\frac{\partial}{\partial t}f_\sigma \leq \Delta
f_\sigma+\frac{4(1-\sigma)}{|H|}\langle\nabla |H|,\nabla
f_\sigma\rangle-\frac{2\epsilon_\nabla}{|H|^{2(1-\sigma)}}|\nabla
H|^2+2\sigma |A|^2f_\sigma+\frac{C}{|H|^{2(1-\sigma)}}+Cf_\sigma.
\end{equation}
\end{prop}

\begin{proof}
By the definition of $f_\sigma$, we have
\begin{equation}\label{evo-f_sigma-1}
\frac{\partial}{\partial t}f_\sigma=\frac{ \frac{\partial}{\partial
t}|\mathring{A}|^2}{|H|^{2(1-\sigma)}}-\frac{(1-\sigma)|\mathring{A}|^2\frac{\partial}{\partial
t}|H|^2}{|H|^{2(2-\sigma)}}.
\end{equation}

 If we put $a=\frac{1}{n}$ and $b=0$ in (\ref{evo-Q}), then get the
 evolution equation of $|\mathring{A}|^2$. From this and
 (\ref{evu-|H|^2}) we have
\begin{equation}\label{evo-f_sigma-2}
\begin{split}
\frac{\partial}{\partial
t}f_\sigma=&\frac{1}{|H|^{2(1-\sigma)}}\Big( \Delta
|\mathring{A}|^2-2|\nabla \mathring{A}|^2 +2R_1-\frac{2}{n}R_2+{\rm
P}_\frac{1}{n} \Big)\\
&-\frac{(1-\sigma)|\mathring{A}|^2}{|H|^{2(1-\sigma)}}\Big(
\Delta|H|^2-2|\nabla H|^2 +2R_2
+\sum_{k,\alpha,\beta}\bar{R}_{k\alpha k\beta}H^\alpha H^\beta\Big).
\end{split}
\end{equation}

The Laplacian of $f_\sigma$ can be computed as
\begin{equation}\label{laplacian-f_sigma}\begin{split}
\Delta f_\sigma
=&\frac{\Delta|\mathring{A}|^2}{|H|^{2(1-\sigma)}}-\frac{(1-\sigma)|\mathring{A}|^2\Delta|H|^2}{|H|^{2(2-\sigma)}}
+\frac{(2-\sigma)(1-\sigma)|\mathring{A}|^2|\nabla
|H|^2|^2}{|H|^{2(3-\sigma)}}\\
&-\frac{2(1-\sigma)\langle \nabla |\mathring{A}|^2,\nabla |H|^2
\rangle}{|H|^{2(2-\sigma)}}.
\end{split}\end{equation}

On the other hand, we have
\begin{equation}\label{1}\begin{split}
-\frac{2(1-\sigma)\langle \nabla |\mathring{A}|^2,
\nabla|H|^2\rangle}{|H|^{2(2-\sigma)}}=-\frac{2(1-\sigma)}{|H|^2}
\langle  \nabla |H|^2,\nabla f_\sigma
\rangle-\frac{8(1-\sigma)^2}{|H|^4}f_\sigma |H|^2 |\nabla |H||^2.
\end{split}\end{equation}
Hence
\begin{equation}\label{evo-f_sigma-3}
\begin{split}
\Big(\frac{\partial}{\partial t}- \Delta\Big)
f_\sigma=&\frac{2(1-\sigma)}{|H|^{2}}\langle \nabla |H|^2,\nabla
f_\sigma\rangle-\frac{2}{|H|^{2(1-\sigma)}}\Big(|\nabla A|^2
-\frac{|A|^2}{|H|^2}|\nabla H|^2\Big)\\
&-\frac{2\sigma|\mathring{A}|^2}{|H|^{2(2-\sigma)}}|\nabla
H|^2-\frac{4\sigma(1-\sigma)}{|H|^4}f_\sigma |H|^2|\nabla |H||^2\\
&+\frac{2\sigma R_2
f_\sigma}{|H|^2}+\frac{2}{|H|^{2(1-\sigma)}}\Big(R_1-\frac{|A|^2}{|H|^2}R_2\Big)\\
&+\frac{1}{|H|^{2(1-\sigma)}}{\rm
P}_{\frac{1}{n}}-\frac{2(1-\sigma)|\mathring{A}|^2}{|H|^{2(2-\sigma)}}\sum_{k,\alpha,\beta}\bar{R}_{k\alpha
k\beta}H^\alpha H^\beta.
\end{split}
\end{equation}

From Lemma \ref{lem-gradient-ineq} we have
\begin{equation}\label{2}
\begin{split}
|\nabla A|^2 -\frac{|A|^2}{|H|^2}|\nabla H|^2\geq&
\Big(\frac{3}{n+2}-\eta-a \Big)|\nabla H|^2-C(n,d,K_1,K_2,\eta)\\
=&\epsilon_\nabla |\nabla H|^2-C(n,d,K_1,K_2,\eta).
\end{split}\end{equation}
Here $a<\frac{4}{3n}$ for $n=2,\,3$ and $a<\frac{1}{n-1}$ for $n\geq
4$, and we choose positive constant $\eta$ depending only on $n$
such that $\epsilon_\nabla=\frac{3}{n+2}-\eta-a>0$. We also have the
following estimates.
$$R_2\leq |A|^2|H|^2,$$
$$R_1-\frac{|A|^2}{|H|^2}R_2 \leq0,$$
$${\rm P}_{\frac{1}{n}}\leq C|\mathring{A}|^2+C,$$ and
$$\sum_{k,\alpha,\beta}\bar{R}_{k\alpha k\beta}H^\alpha H^\beta\leq
C|H|^2,$$ where $C$ is a positive constant depending on $n$, $d$,
$K_1$, $K_2$ and $L$. It follows from these estimates,
(\ref{evo-f_sigma-3}) and (\ref{2}) that
\begin{equation*}
\begin{split}
\Big(\frac{\partial}{\partial t}- \Delta\Big) f_\sigma\leq
&\frac{4(1-\sigma)}{|H|}\langle \nabla |H| ,\nabla
f_\sigma\rangle-\frac{2\epsilon_\nabla}{|H|^{2(1-\sigma)}} |\nabla
H|^2 \\
& +2\sigma|A|^2f_\sigma+\frac{C}{|H|^{2(1-\sigma)}}+Cf_\sigma.
\end{split}
\end{equation*}
This completes the proof.
\end{proof}

To handle the reaction term $2\sigma|A|^2f_\sigma$, we need to
compute the Laplacian of $|\mathring{A}|^2$. As in
\cite{Andrews-Baker}, we have
\begin{equation}\label{laplacian-A0^2}
\begin{split}
\frac{1}{2}\Delta |\mathring{A}|^2\geq&|\nabla
\mathring{A}|^2+\langle \mathring{h}_{ij}, \nabla_i \nabla_j
H\rangle+Z-C|H|^2-C
\end{split}
\end{equation}
for some positive constant $C$ depending on $n,\ d, \ K_1,\ K_2$ and
$L$. Here
\begin{equation}
Z=\sum_{i,j,p,\alpha,\beta}H^\alpha h_{ip}^\alpha h_{pj}^{\beta}
h_{ij}^\beta-\sum_{\alpha,\beta}\Big(\sum_{i,j}h_{ij}^\alpha
h_{ij}^\beta\Big)^2-\sum_{i,j,\alpha,\beta}\Big(\sum_{p}\Big(h^\alpha_{ip}h^\beta_{jp}-h^\alpha_{jp}h^\beta_{ip}\Big)\Big)^2.
\end{equation}

\begin{prop}
There is a positive constant $C$  depending on $n,\ d,\ K_1, \ K_2,
\ L$ and $M_0$ such that for any $p\geq2$ and $\eta>0$, the
following inequality holds.
\begin{equation}\label{integral-ineq-5}
\begin{split}
\int_{M_t}|H|^2f_\sigma^p  \leq \frac{2p\eta+C}{\epsilon}
\int_{M_t}\frac{f_\sigma^{p-1}}{|H|^{2(1-\sigma)}}|\nabla
H|^2+\frac{p-1}{\epsilon\eta}\int_{M_t}f_\sigma^{p-2}|\nabla
f_\sigma|^2+C^p.
\end{split}
\end{equation}

\end{prop}

\begin{proof}
From (\ref{laplacian-f_sigma}) and (\ref{laplacian-A0^2}), we have
\begin{equation}\label{laplacian-f_sigma-ineq}
\begin{split}
\Delta f_\sigma\geq&\frac{2}{|H|^{2(1-\sigma)}}\langle
\mathring{h}_{ij},\nabla_i \nabla_j H
\rangle+\frac{2}{|H|^{2(1-\sigma)}}|\nabla\mathring{A}|^2+\frac{2}{
|H|^{2(1-\sigma)}}Z-\frac{2(C|H|^2+C)}{|H|^{2(1-\sigma)}}\\
&-\frac{2(1-\sigma)}{|H|}f_\sigma \Delta
|H|-\frac{2(1-\sigma)}{|H|^2}f_\sigma|\nabla
|H||^2+\frac{4(2-\sigma)(1-\sigma)}{|H|^2}f_\sigma|\nabla |H||^2\\
&-\frac{4(1-\sigma)}{|H|}\langle \nabla|H|,\nabla f_\sigma
\rangle-\frac{8(1-\sigma)^2}{|H|^2}f_{\sigma}|\nabla |H||^2\\
=&\frac{2}{|H|^{2(1-\sigma)}}\langle \mathring{h}_{ij},\nabla_i
\nabla_j H
\rangle+\frac{2}{|H|^{2(1-\sigma)}}|\nabla\mathring{A}|^2+\frac{2}{
|H|^{2(1-\sigma)}}Z\\
&-\frac{4(1-\sigma)}{|H|}\langle \nabla|H|,\nabla f_\sigma
\rangle-\frac{2(1-\sigma)}{|H|}f_\sigma \Delta |H|\\
&-\frac{2(1-\sigma)(1-2\sigma)}{|H|^2}f_{\sigma}|\nabla |H||^2
-\frac{C}{|H|^{-2\sigma}}-\frac{C}{|H|^{2(1-\sigma)}}.
\end{split}
\end{equation}

Multiplying both sides of (\ref{laplacian-f_sigma-ineq}) by
$f_\sigma^{p-1}$ and integrating over $M_t$ we obtain
\begin{equation}\label{integral-ineq-1}
\begin{split}
0\geq& (p-1) \int_{M_t}f_\sigma^{p-2}|\nabla
f_\sigma|^2+2\int_{M_t}\frac{f_\sigma^{p-1}}{|H|^{2(1-\sigma)}}\langle
\mathring{h}_{ij},\nabla_i \nabla_j H
\rangle+2\int_{M_t}\frac{f_\sigma^{p-1}}{|H|^{2(1-\sigma)}}|\nabla\mathring{A}|^2\\
&+2\int_{M_t}\frac{f_\sigma^{p-1}}{
|H|^{2(1-\sigma)}}Z-4(1-\sigma)\int_{M_t}\frac{f_\sigma^{p-1}}{|H|}\langle
\nabla|H|,\nabla f_\sigma
\rangle-2(1-\sigma)\int_{M_t}\frac{f_\sigma^p}{|H|}f_\sigma \Delta
|H|\\
&-2(1-\sigma)(1-2\sigma)\int_{M_t}\frac{f_\sigma^p}{|H|^2} |\nabla
|H||^2-C\int_{M_t}\frac{f_\sigma^{p-1}}{|H|^{-2\sigma}}-C\int_{M_t}\frac{f_{\sigma}^p}{|H|^{2(1-\sigma)}}.
\end{split}
\end{equation}

The first term on the right hand side of (\ref{integral-ineq-1}) is
nonnegative. For the second term, we have the following computation.
\begin{equation}\label{2nd-term}
\begin{split}
&2\int_{M_t}\frac{f_\sigma^{p-1}}{|H|^{2(1-\sigma)}}\langle
\mathring{h}_{ij},\nabla_i \nabla_j H \rangle \\=&-2\int_{M_t}
\Big\langle \nabla_i\Big(
\frac{f_\sigma^{p-1}}{|H|^{2(1-\sigma)}}\mathring{h}_{ij}
\Big),\nabla_j H \Big\rangle \\
=& -2\int_{M_t}\frac{(p-1)f_\sigma^{p-2}}{|H|^{2(1-\sigma)}}\langle
\nabla_if_\sigma \mathring{h}_{ij},\nabla_jH
\rangle+4\int_{M_t}\frac{(1-\sigma)f_\sigma^{p-1}}{|H|^{3-2\sigma}}\langle
\nabla_i|H| \mathring{h}_{ij},\nabla_jH \rangle\\
&-\frac{2(n-1)}{n}\int_{M_t}\frac{f_\sigma^{p-1}}{|H|^{2(1-\sigma)}}|\nabla
H|^2-2\int_{M_t}\frac{f_\sigma^{p-1}}{|H|^{2(1-\sigma)}} \langle
\sum_{i,j,\alpha}\bar{R}_{jii\alpha}e_\alpha,\nabla_jH \rangle.
\end{split}
\end{equation}

We also have
\begin{equation}\label{3}
\begin{split}
&-2(1-\sigma)\int_{M_t}\frac{f_\sigma^p}{|H|}f_\sigma \Delta
|H|\\
=&2(1-\sigma)\int_{M_t} \Big\langle \nabla
\Big(\frac{f_\sigma^p}{|H|}\Big),\nabla|H| \Big\rangle\\
=&2(1-\sigma)\int_{M_t}\frac{pf_\sigma^{p-1}}{|H|}\langle \nabla
f_\sigma,
\nabla|H|\rangle-2(1-\sigma)\int_{M-t}\frac{f_\sigma^p}{|H|^2}|\nabla|H||^2.
\end{split}
\end{equation}

Combining (\ref{integral-ineq-1}), (\ref{2nd-term}) and (\ref{3})
implies
\begin{equation}\label{integral-ineq-2}
\begin{split}
2\int_{M_t}\frac{f_\sigma^{p-1}}{|H|^{2(1-\sigma)}}Z\leq
&2(p-1)\int_{M_t}\frac{f_\sigma^{p-2}}{|H|^{2(1-\sigma)}}\langle
\nabla_if_\sigma \mathring{h}_{ij},\nabla_jH \rangle\\
&-4(1-\sigma)\int_{M_t}\frac{f_\sigma^{p-1}}{|H|^{3-2\sigma}}\langle
\nabla_i|H| \mathring{h}_{ij},\nabla_jH \rangle\\
&+\frac{2(n-1)}{n}\int_{M_t}\frac{f_\sigma^{p-1}}{|H|^{2(1-\sigma)}}|\nabla
H|^2\\
&-2(1-\sigma)(p-2)\int_{M_t}\frac{ f_\sigma^{p-1}}{|H|}\langle
\nabla |H|,\nabla f_\sigma \rangle
\\
&+2(1-\sigma)\int_{M_t}\frac{f\sigma^p}{|H|^2}|\nabla |H| |^2\\
&+C\int_{M_t}\frac{f_\sigma^{p-1}}{|H|^{-2\sigma}}+C\int_{M_t}\frac{f_{\sigma}^p}{|H|^{2(1-\sigma)}}\\
&+C\int_{M_t}\frac{f_\sigma^{p-1}}{|H|^{2(1-\sigma)}} |\nabla H|.
\end{split}
\end{equation}
Here $C$ is a positive constant depending on $n,\ d, \ K_1,\ K_2$
and $L$.

Notice that $f_\sigma\leq C|H|^{2\sigma}$ and $|\mathring{A}|^2\leq
f_\sigma|H|^{2(1-\sigma)}$, and we can pick $\sigma\in (0,1)$
sufficiently small. Also notice that the initial pinching condition
is preserved and implies that for $t\geq0$, there holds
\begin{equation*}|H|^2\geq C>0
\end{equation*}
 for some positive constant $C$ depending on  $n,\ d, \ K_1,\ K_2$
and $L$. This implies that
\begin{equation}\label{4}
\frac{f_\sigma^{p-1}}{|H|^{-2\sigma}}=
\frac{|H|^{\frac{2(p-1)}{p}}f_\sigma^{p-1}}{|H|^{2(1-\frac{1}{p})-2\sigma}}\leq
C|H|^{\frac{2(p-1)}{p}}f_\sigma^{p-1},
\end{equation}
\begin{equation}\label{5}
\frac{f_\sigma^{p-1}}{|H|^{2(1-\sigma)}}=\frac{|H|^{\frac{2(p-1)}{p}}f_\sigma^{p-1}}{|H|^{\frac{2(p-1)}{p}+2(1-\sigma)}}\leq
C|H|^{\frac{2(p-1)}{p}}f_\sigma^{p-1},
\end{equation}
and
\begin{equation}\label{6}
\frac{f_\sigma^{p-1}}{|H|^{2(1-\sigma)}} |\nabla
H|\leq\bar{\varepsilon} \frac{f_\sigma^{p-1}}{|H|^{2(1-\sigma)}}+
\bar{\varepsilon}^{-1} \frac{f_\sigma^{p-1}}{|H|^{2(1-\sigma)}}
|\nabla H|^2,
\end{equation}
for any positive constant $\bar{\varepsilon}$

 By using
Young's inequality, we have the following
\begin{equation}\label{7}|H|^{\frac{2(p-1)}{p}}f_\sigma^{p-1}\leq
\tilde{\varepsilon}^{\frac{p}{p-1}}
|H|^2f_\sigma^p+\tilde{\varepsilon}^{-p}
\end{equation}
for arbitrary positive constant $\tilde{\varepsilon}$.

From (\ref{4})-(\ref{7}), the last two line of
(\ref{integral-ineq-2}) is not bigger than
\begin{equation}\label{last-2-line}
C (1+\bar{\varepsilon})\tilde{\varepsilon}^{\frac{p}{p-1}}
\int_{M_t}|H|^2f_\sigma^p+C\tilde{\varepsilon}^{-p}{\rm
Vol}(M_0)+C\bar{\varepsilon}^{-1}\int_{M_t}\frac{f_\sigma^{p-1}}{|H|^{2(1-\sigma)}}
\end{equation}
for some positive constant $C$ depending on  $n, \ d,\ K_1,\ K_2$
and $L$.

On the other hand, we have the following inequalities as in
\cite{Andrews-Baker}.
\begin{equation}\label{8}
\begin{split}
&2(p-1)\int_{M_t}\frac{f_\sigma^{p-2}}{|H|^{2(1-\sigma)}}\langle
\nabla_if_\sigma \mathring{h}_{ij},\nabla_jH \rangle\\
&\leq \frac{p-1}{\eta}\int_{M_t}f_\sigma^{p-2}|\nabla
f_\sigma|^2+(p-1)\eta
\int_{M_t}\frac{f_\sigma^{p-1}}{|H|^{2(1-\sigma)}}|\nabla H|^2,
\end{split}
\end{equation}
\begin{equation}\label{9}
\begin{split}
-4(1-\sigma)\int_{M_t}\frac{f_\sigma^{p-1}}{|H|^{3-2\sigma}}\langle
\nabla_i|H| \mathring{h}_{ij},\nabla_jH \rangle\leq
4\int_{M_t}\frac{f_\sigma^{p-1}}{|H|^{2(1-\sigma)}}|\nabla H|^2,
\end{split}
\end{equation}
\begin{equation}\label{10}
\begin{split}
&-2(1-\sigma)(p-2)\int_{M_t}\frac{ f_\sigma^{p-1}}{|H|}\langle
\nabla |H|,\nabla f_\sigma \rangle\\
&\leq\frac{p-2}{\mu}\int_{M_t}f_\sigma^{p-2}|\nabla
f_\sigma|^2+(p-2)\mu\int_{M_t}\frac{f_\sigma^{p-1}}{|H|^{2(1-\sigma)}}|\nabla
H|^2,
\end{split}
\end{equation}
\begin{equation}\label{11}
2(1-\sigma)\int_{M_t}\frac{f_\sigma^p}{|H|^2}|\nabla
|H||^2\leq2\int_{M_t}\frac{f_\sigma^{p-1}}{|H|^{2(1-\sigma)}}|\nabla
H|^2.
\end{equation}

Combining (\ref{integral-ineq-2}), (\ref{last-2-line})-(\ref{11})
implies
\begin{equation}\label{integral-ineq-3}
\begin{split}
&2\int_{M_t}\frac{f_\sigma^{p-1}}{|H|^{2(1-\sigma)}}Z\\
&\leq\bigg(6+ \frac{2(n-1)}{n}+ (p-1)\eta +(p-2)\mu+
C\bar{\varepsilon}^{-1}\bigg)
\int_{M_t}\frac{f_\sigma^{p-1}}{|H|^{2(1-\sigma)}}|\nabla H|^2\\
&{\
 \ \ }+\bigg(\frac{p-1}{\eta}+\frac{p-2}{\mu}\bigg)\int_{M_t}f_\sigma^{p-2}|\nabla
f_\sigma|^2+C(1+\bar{\varepsilon})\tilde{\varepsilon}^{\frac{p}{p-1}}
\int_{M_t}|H|^2f_\sigma^p+C\tilde{\varepsilon}^{-p}{\rm Vol}(M_0).
\end{split}
\end{equation}

From Lemma 4 in \cite{Andrews-Baker} we have $Z\geq \epsilon
|\mathring{A}|^2|H|^2$ for some positive constant $\epsilon$. Then
from (\ref{integral-ineq-3}) we have
\begin{equation}\label{integral-ineq-4}
\begin{split}
&\Big(2\epsilon-C(1+\bar{\varepsilon})\tilde{\varepsilon}^{\frac{p}{p-1}}\Big)\int_{M_t}|H|^2f_\sigma^p\\
&\leq\bigg(6+ \frac{2(n-1)}{n}+ (p-1)\eta +(p-2)\mu+
C\bar{\varepsilon}^{-1}\bigg)
\int_{M_t}\frac{f_\sigma^{p-1}}{|H|^{2(1-\sigma)}}|\nabla H|^2\\
&{\
 \ \ }+\bigg(\frac{p-1}{\eta}+\frac{p-2}{\mu}\bigg)\int_{M_t}f_\sigma^{p-2}|\nabla
f_\sigma|^2+C\tilde{\varepsilon}^{-p}{\rm Vol}(M_0).
\end{split}
\end{equation}
Now we put $\bar{\varepsilon}=1$,
$\tilde{\varepsilon}=\Big(\frac{\epsilon}{2C}\Big)^{\frac{p-1}{p}}$,
and let $\eta=\mu$. Then (\ref{integral-ineq-4}) implies
\begin{equation*}
\begin{split}
\int_{M_t}|H|^2f_\sigma^p  \leq \frac{2p\eta+C}{\epsilon}
\int_{M_t}\frac{f_\sigma^{p-1}}{|H|^{2(1-\sigma)}}|\nabla
H|^2+\frac{p-1}{\epsilon\eta}\int_{M_t}f_\sigma^{p-2}|\nabla
f_\sigma|^2+C^p.
\end{split}
\end{equation*}
Here $C$ is a positive constant depending on $n,\ d,\ K_1, \ K_2, \
L$ and $M_0$.
\end{proof}

\begin{prop}\label{16}For sufficiently small $\sigma$ and large $p$, there holds
\begin{equation}\label{evo-int_M f_sigma^p3}
\begin{split}
\frac{d}{d t}\int_{M_t}f_\sigma^p\leq Cp\int_{M_t}f_\sigma^p+C^p,
\end{split}
\end{equation}
where $C$  is a  positive constant depending on $n,\ d,\ K_1, \ K_2,
\ L$ and $M_0$.
\end{prop}
\begin{proof}
We first compute the evolution equation of $\int_{M_t}f_\sigma^p$ as
follows.
\begin{equation}\label{evo-int_M f_sigma^p}
\begin{split}
\frac{d}{d
t}\int_{M_t}f_\sigma^p\leq&\int_{M_t}pf_\sigma^{p-1}\frac{\partial
f_\sigma}{\partial t}\\
\leq&-p(p-1)\int_{M_t}f_\sigma^{p-2} | \nabla
f_\sigma|^2+4(1-\sigma)p\int_{M_t}\frac{f_\sigma^{p-1}}{|H|}|\nabla|H|||\nabla
f_\sigma|\\
&-2p\epsilon_\nabla
\int_{M_t}\frac{f_\sigma^{p-1}}{|H|^{2(1-\sigma)}}|\nabla
H|^2+2p\sigma\int_{M_t}|H|^2f_\sigma^p\\
&+Cp\int_{M_t}\frac{f_\sigma^{p-1}}{|H|^{2(1-\sigma)}}+Cp\int_{M_t}f_\sigma^p.
\end{split}
\end{equation}
The second integral on the right hand side may be estimated as
\begin{equation}\label{12}
4(1-\sigma)p\int_{M)_t}\frac{f_\sigma^{p-1}}{|H|}|\nabla|H|||\nabla
f_\sigma|\leq \frac{2p}{\rho}\int_{M_t}f_\sigma^{p-2}|\nabla
f_\sigma|^2+2p\rho\int_{M_t}\frac{f_\sigma^{p-1}}{|H|^{2(1-\sigma)}}|\nabla
H|^2.\end{equation} The last second term can be estimated  from
(\ref{5}) and (\ref{7}) as
\begin{equation}\label{13}
Cp\int_{M_t}\frac{f_\sigma^{p-1}}{|H|^{2(1-\sigma)}}\leq
\bar{C}\tilde{\tilde{\varepsilon}}^{\frac{p}{p-1}}
p\int_{M_t}|H|^2f_\sigma^p+\bar{C}p\tilde{\tilde{\varepsilon}}^{-p}{\rm
Vol}(M_0),
\end{equation}
for a  positive constant $\bar{C}$  depending on $n,\ d,\ K_1, \
K_2, \ L$ and $M_0$. Combining (\ref{integral-ineq-5}),
(\ref{evo-int_M f_sigma^p})-(\ref{13}) implies
\begin{equation}\label{evo-int_M f_sigma^p2}
\begin{split}
\frac{d}{d
t}\int_{M_t}f_\sigma^p\leq&\int_{M_t}pf_\sigma^{p-1}\frac{\partial
f_\sigma}{\partial t}\\
\leq&-p(p-1)\bigg(1-\frac{2}{\rho(p-1)}-\frac{2\sigma}{\epsilon\eta}
-\frac{\bar{C}\tilde{\tilde{\varepsilon}}^{\frac{p}{p-1}}}{\epsilon
\eta}\bigg)\int_{M_t}f_\sigma^{p-2} | \nabla
f_\sigma|^2 \\
&-2p\epsilon_\nabla\bigg(1-\frac{\rho}{\epsilon_\nabla}
-\frac{(2p\eta+C)\sigma}{\epsilon\epsilon_\nabla} +
\frac{(2p\eta+C)\bar{C}\tilde{\tilde{\varepsilon}}^{\frac{p}{p-1}}}{2\epsilon\epsilon_\nabla}\bigg)
\int_{M_t}\frac{f_\sigma^{p-1}}{|H|^{2(1-\sigma)}}|\nabla
H|^2\\
&
+Cp\int_{M_t}f_\sigma^p+\Big(2p\sigma+\bar{C}\tilde{\tilde{\varepsilon}}^{\frac{p}{p-1}}p\Big)
C^p+\bar{C}p\tilde{\tilde{\varepsilon}}^{-p}{\rm Vol}(M_0).
\end{split}
\end{equation}
 Now we pick $\rho=\frac{\epsilon_\nabla}{4}$,
$\eta=\frac{1}{p}$, $\tilde{\tilde{\varepsilon}}=\Big(\min\{
\frac{\epsilon\epsilon_\nabla}{2(2+C)},\frac{\epsilon}{4\bar{C}p}\}\Big)^{\frac{p-1}{p}}$,
and let $p\geq\frac{32}{\epsilon_\nabla}+1$ and $\sigma\leq
\min\{\frac{\epsilon\epsilon_\nabla}{4(2+C)},
\frac{\epsilon}{8p}\}$. Then (\ref{evo-int_M f_sigma^p2}) reduces to
\begin{equation*}
\begin{split}
\frac{d}{d t}\int_{M_t}f_\sigma^p\leq Cp\int_{M_t}f_\sigma^p+C^p,
\end{split}
\end{equation*}
for a  positive constant $C$  depending on $n,\ d,\ K_1, \ K_2, \ L$
and $M_0$.
\end{proof}

To give an upper bound of $\int_{M_t}f_\sigma^p$, we need to show
that the maximal existence time of the mean curvature flow under the
initial curvature pinching condition is finite.
\begin{lemma}\label{17}
Under the initial curvature pinching condition, the mean curvature
flow has a smooth solution on a finite maximal time interval $[0,T)$
and $\max_{M_t}|A|\rightarrow \infty$ as $t\rightarrow T$.
\end{lemma}
\begin{proof}
We first show that if the second fundamental form is uniformly
bounded by a positive constant $K$ in a finite time interval
$[0,T)$, then the mean curvature flow can be extended over the time.
For any fixed $x\in M$, any $\tau,\varrho\in [0,T)$ such that
$\tau<\varrho$, $F(x,t)$, $t\in [\tau,\varrho]$ is a segment in $N$
connecting $F(x,\tau)$ and $F(x,\varrho)$. We have
\begin{equation}\label{finite-dist}
{\rm dist }(F(x,\tau),F(x,\varrho))\leq\int_{\tau}^{\varrho}\Big|
\frac{\partial}{\partial t}F(x,t)
\Big|dt=\int_{\tau}^{\varrho}|H(x,t)|dt\leq CT<\infty.
\end{equation}
Hence $F(x,t)$ converges uniformly to some continuous limit function
$F(x,T)$.  We want to show that $F(x,T)$ is a smooth immersion. To
do this, we only have to establish uniform bounds  for all the
covariant derivatives of the second fundamental form on $M_t$, $t\in
[0,T)$. Since $M_t$ stays in a compact region on $N$ in view of
(\ref{finite-dist}), we have $\max_{0\leq l\leq
m}|\bar{\nabla}^m\bar{R}|\Big|_{M_t}\leq C_m$ for positive constants
$C_m$ independent of $t$.

First, by the evolution of the second fundamental form, as
Proposition 7.1 in \cite{WaM1}  for example, and induction on $m$,
we have the following evolution equations
\begin{equation}\label{evo-derivatives}
\begin{split}
\frac{\partial}{\partial t}|\nabla^m A|^2\leq &\Delta|\nabla^m
A|^2-2|\nabla ^{m+1}A|^2\\
&+C(n,m)\Big\{ \sum_{i+j+k=m} |\nabla^i A||\nabla^j A||\nabla^k
A||\nabla^m A|\\
& +\bar{C}_m \sum_{i\leq m} |\nabla^i A||\nabla^m A|
+\bar{\bar{C}}_m|\nabla^m A| \Big\}.
\end{split}
\end{equation}

For any $\tau\in [0,T)$, we consider the mean curvature flow on the
time interval $[0,\tau]$. Consider the function $G=t|\nabla
A|^2+|A|^2$, which is bounded by $K$ at $t=0$. Differentiating the
function we get for some positive constant $C$ independent of $t$
and $\tau$
\begin{equation}\label{evo-derivatives-G}
\begin{split}
\frac{\partial}{\partial t}G\leq& \Delta G+|\nabla
A|^2-2(t|\nabla^2A|^2+|\nabla
A|^2)\\
&+Ct(|A|^2|\nabla A|^2+|\nabla A|^2+|A||\nabla A|+|\nabla A|)\\
&+C(|A|^4+|A|^2+|A|).
\end{split}
\end{equation}

Since we have already assumed that the second fundamental form is
uniformly bounded on $[0,T)$ by $K$, (\ref{evo-derivatives-G}) may
be reduced to
\begin{equation}\label{evo-derivatives-G2}
\begin{split}
\frac{\partial}{\partial t}G\leq& \Delta G+(Ct-1)|\nabla A|^2+C,
\end{split}
\end{equation}
where $C$ is a positive constant independent of $t$ and $\tau$. For
$t\in [0,\frac{1}{C}]$, where $C$ is as in
(\ref{evo-derivatives-G2}),  we obtain from
(\ref{evo-derivatives-G2})
\begin{equation}\label{evo-derivatives-G3}
\begin{split}
\frac{\partial}{\partial t}G\leq& \Delta G  +C,
\end{split}
\end{equation}
which implies by the maximum principle $G\leq  K+Ct$. Hence $|\nabla
A|^2\leq \frac{G}{t}\leq \frac{K}{t}+C$ for $t\in (0,\frac{1}{C}]$.
If $t>\frac{1}{C}$, the same argument on time interval
$[t-\frac{1}{C},t]$ yields $|\nabla A|^2\leq \tilde{C}$ for some
$\tilde{C}$ independent of $t$ and $\tau$. Since $\tau\in [0,T)$ is
arbitrary, we have $|\nabla A|^2\leq C(1+\frac{1}{t})$ for $t\in
(0,T)$, where $C$ a constant independent of $t$.

Now we assume that there are positive constants $C_k$ independent of
$t$ such that $|\nabla^k A|^2\leq C_{k}(1+\frac{1}{t^k})$ holds for
$k=1,2,\cdots,m-1$ and $t\in (0,T)$. Using Young's inequality,
(\ref{evo-derivatives}) implies
\begin{equation}\label{evo-derivatives,}
\begin{split}
\frac{\partial}{\partial t}|\nabla^k A|^2\leq &\Delta|\nabla^k
A|^2-2|\nabla ^{k+1}A|^2
+\tilde{C}_k\bigg[\Big(1+\frac{1}{t^k}\Big)^{\frac{1}{2}}|\nabla^k
A|+|\nabla^k A|^2+1\bigg]\\
\leq&\Delta|\nabla^k A|^2-2|\nabla ^{k+1}A|^2
+\tilde{C}_k\Big(|\nabla^k A|^2+1+\frac{1}{t^k}\Big),
\end{split}
\end{equation}
where $\tilde{C}_k$ is a positive constant depending on $k$ and
others but independent of $t$.

 Consider $G=t^m|\nabla^m
A|^2+mt^{m-1}|\nabla^{m-1}A|^2$. Differentiating $G$ with respect to
$t$ gives
\begin{equation}\label{evo-derivatives-Gm}
\begin{split}
\frac{\partial}{\partial t}G\leq& mt^{m-1}|\nabla ^m A|^2
+t^m\Big[\Delta|\nabla^mA|^2 +\tilde{C}_{m} \Big(|\nabla^m
A|^2+1+\frac{1}{t^m}\Big) \Big]\\
&+m\Big\{ (m-1)t^{m-2}|\nabla^{m-1}A|^2\\
& +t^{m-1}\Big[\Delta|\nabla^{m-1}A|^2-2|\nabla^{m}A|^2
+\tilde{C}_{m-1} \Big(|\nabla^{m-1} A|^2+1+\frac{1}{t^{m-1}}\Big)
\Big]\Big\}\\
\leq& \Delta G +(\tilde{C}_{m}t-m)t^{m-1}|\nabla ^m A|^2 +\hat{C}_m,
\end{split}
\end{equation}
provided $t\leq 1$, for some positive constant $\hat{C}_m$ depending
on $m$ and others but independent of $t$. Hence for $t\in(0,
\min\{1,\frac{m}{\tilde{C}_{m}}\})$, by maximum principle,
$|\nabla^m A|^2\leq \frac{G}{t^m}\leq \frac{\hat{C}_m}{t^{m-1}}\leq
C_m(1+\frac{1}{t^m})$ for some positive constant ${C}_m$ depending
on $m$ and others but independent of $t$. For
$t>\min\{1,\frac{m}{\tilde{C}_{m}}\}$, the same argument gives the
bound of $|\nabla^mA|^2$. Hence we have proved that
$|\nabla^mA|^2\leq C_m(1+\frac{1}{t^m})$  for $t\in(0,T)$ for
constants ${C}_m$ independent of $t$.

Now we prove that, under the initial curvature pinching condition,
the second fundamental form of the submanifold will blow up in
finite time. Consider the function $Q=|A|^2-a|H|^2+b(t)$ where
$a=\frac{4}{3n}$  for $n=2,3$ and $a=\frac{1}{n-1}$  for $n\geq4$
and $b(t)$ is a function of $t$ with $b(0)$ such that $Q(0)<0$. By
the initial pinching assumption, we may pick
$b(0)=b_\varepsilon>b_1$. We will show that $Q<0$ is preserved for a
suitable $b(t)$. Suppose not, then there is a first time such that
$Q=0$ at a point. Then at this point, we have the following estimate
as (\ref{14})
\begin{equation}\label{15}
\begin{split} \Big(\frac{\partial}{\partial t}-\Delta\Big) Q\leq &
-\frac{2nab}{n(a-\frac{1}{n})}|\mathring{A}|_H^2-\frac{4b}{n(a-\frac{1}{n})}|\mathring{A}|_I^2-\frac{2b^2}{n(a-\frac{1}{n})}\\
&+C_1|\mathring{A}|_H^2+C_2|\mathring{A}|_I^2+C_3 b+C_4+b'(t).
\end{split}\end{equation}

Now we pick $b(t)$ such that
$b'(t)=\frac{2b^2}{n(a-\frac{1}{n})}-C_3 b-C_4$. Then
$\Big(\frac{\partial}{\partial t}-\Delta\Big) Q<0$, which implies
that $Q<0$ is preserved along the mean curvature flow. On the other
hand, it is easy to check that $b(t)$ is monotone increasing and
becomes unbounded as $t\rightarrow t_0$ for some $t_0<\infty$. Hence
$|A|^2$ becomes unbounded as $t\rightarrow t_0$. This completes the
proof of the lemma.

\end{proof}

Now from Proposition \ref{16} and Lemma \ref{17}, we see that
$\int_{M_t}f_\sigma^p\leq C$ holds for $t\in [0,T)$ with a positive
constant $C$ independent of $t$. We also have a Sobolev inequality
for $M_t$ (see \cite{HS}). Hence we may apply the same argument as
in \cite{H1} and \cite{H2} to derive a bound for $f_\sigma$ if
$\sigma$ is small enough. This completes the proof of Theorem
\ref{thm-pinching-|A0|^2}.

\section{The gradient estimate of the mean curvature}

To compare the mean curvature of the submanifold at different
points, we need to give an estimate of the gradient of the mean
curvature.
\begin{theorem}\label{thm-gradient}
For any $\eta>0$, there is a constant $C_\eta<\infty$ independent of
$t$ such that
\begin{equation}
|\nabla H|^2\leq \eta |H|^4+C_\eta.
\end{equation}
\end{theorem}

\begin{proof}
First we have the following inequalities by using a similar
computation as in \cite{Andrews-Baker} and \cite{H2}
\begin{equation}\label{18}
\frac{\partial}{\partial t}|\nabla H|^2\leq \Delta |\nabla
H|^2-2|\nabla^2 H|^2+C|H|^2|\nabla A|^2+C|\nabla A|^2+C|H|^2,
\end{equation}
\begin{equation}\label{19}
\frac{\partial}{\partial t}| H|^4 \geq \Delta|H|^4 -12|H|^2|\nabla
H|^2+\frac{2}{n}|H|^6-C,
\end{equation}
where $C$ is a positive constant independent of $t$.

For any $N_1,N_2>0$,
\begin{equation}\label{20}
\begin{split}
\frac{\partial}{\partial t} ((N_1+N_2|A|^2)|\mathring{A}|^2)=&
\Delta ((N_1+N_2|A|^2)|\mathring{A}|^2)-2N_2\langle \nabla|A|^2,
\nabla|\mathring{A}|^2 \rangle\\
&-2N_2|\mathring{A}|^2|\nabla A|^2-2(N_1+N_2|A|^2)|\nabla
\mathring{A}|^2\\
&+2N_2(R_1+{\rm
P}_0)|\mathring{A}|^2+2(N_1+N_2|A|^2)(R_1-\frac{R_2}{n}+{\rm
P}_{\frac{1}{n}}).
\end{split}
\end{equation}

The second term on the right can be estimated as follows.
\begin{equation}\label{21}
\begin{split}
-2N_2\langle \nabla|A|^2, \nabla|\mathring{A}|^2
\rangle\leq&8N_2|A||\nabla A||\mathring{A}||\nabla\mathring{A}|\\
\leq&8N_2C_n|H||\nabla A|^2 \sqrt{C_0}|H|^{1-\frac{\delta}{2}}\\
\leq&8N_2 C_n\sqrt{C_0}|\nabla
A|^2(\varrho^\frac{4}{4-\delta}|H|^2+\varrho^{-\frac{4}{\delta}}),
\end{split}
\end{equation}
where $C_n=\frac{2}{\sqrt{3n}}$, $C_0$ and $\delta$ are as in
Theorem \ref{thm-pinching-|A0|^2}, and $\varrho>0$ is an arbitrary
constant. Using Young's inequality and the pinching assumption, we
have $R_1+{\rm P}_0\leq C(|H|^4+1)$ and $R_1-\frac{R_2}{n}+{\rm
P}_{\frac{1}{n}}\leq C|\mathring{A}|^2(|H|^2+1)+C$, where $C$ is
independent of $t$. These together with (\ref{20}) and (\ref{21})
imply
\begin{equation}\label{22}
\begin{split}
\frac{\partial}{\partial t} ((N_1+N_2|A|^2)|\mathring{A}|^2)\leq&
\Delta ((N_1+N_2|A|^2)|\mathring{A}|^2)\\
&+8N_2 C_n\sqrt{C_0}|\nabla
A|^2(\varrho^\frac{4}{4-\delta}|H|^2+\varrho^{-\frac{4}{\delta}})\\
&-2N_2|\mathring{A}|^2|\nabla A|^2-2(N_1+N_2|A|^2)\Big(\frac{n-1}{2n+1}|\nabla A|^2-C\Big)\\
&+2N_2 |\mathring{A}|^2  (C|H|^4+C)\\
& +2(N_1+N_2|A|^2)(C|\mathring{A}|^2(|H|^2+1)+C) \\
\leq&\Delta
((N_1+N_2|A|^2)|\mathring{A}|^2)\\
&-\Big(\frac{2(n-1)}{n(2n+1)}N_2-8N_2 C_n\sqrt{C_0}
\varrho^\frac{4}{4-\delta} \Big)|H|^2|\nabla A|^2\\
&-\Big(\frac{2(n-1)}{2n+1}N_1  -8N_2 C_n\sqrt{C_0}
\varrho^{-\frac{4}{ \delta}}   \Big)|\nabla A|^2\\
&+(2N_2C+2N_2CC_n) |\mathring{A}|^2|H|^4\\
&+C(N_1,N_2)|H|^4+C(N_1,N_2)|H|^2+C(N_1,N_2).
\end{split}
\end{equation}
Here $C$ is some positive constant and $C(N_1,N_2)$ is some positive
constant depending on $N_1$, $N_2$ and others. Choose $\varrho$ such
that $\frac{(n-1)}{n(2n+1)} =8 C_n\sqrt{C_0}
\varrho^\frac{4}{4-\delta}$. Then (\ref{22}) implies
\begin{equation}\label{23}
\begin{split}
\frac{\partial}{\partial t} ((N_1+N_2|A|^2)|\mathring{A}|^2)\leq&
\Delta
((N_1+N_2|A|^2)|\mathring{A}|^2)\\
&-\frac{(n-1)}{n(2n+1)}N_2 |H|^2|\nabla A|^2\\
&-\Big(\frac{2(n-1)}{2n+1}N_1  -C(N_2)  \Big)|\nabla A|^2\\
&+(2N_2C+2N_2CC_n) |\mathring{A}|^2|H|^4\\
&+C(N_1,N_2)|H|^4+C(N_1,N_2)|H|^2+C(N_1,N_2),
\end{split}
\end{equation}
for some positive constant $C(N_2)$ depending on $N_2$ and others
but independent of $N_1$.

Now consider the function $f=|\nabla
H|^2+(N_1+N_2|A|^2)|\mathring{A}|^2-\eta|H|^4$. Then $f$ satisfies
\begin{equation}\label{24}
\begin{split}
\frac{\partial}{\partial t} f\leq& \Delta f
 -\Big(\frac{(n-1)}{n(2n+1)}N_2 -C -\frac{12}{n}\eta \Big)|H|^2|\nabla A|^2\\
&-\Big(\frac{2(n-1)}{2n+1}N_1  -C(N_2) -C \Big)|\nabla A|^2\\
&+(2N_2C+2N_2CC_n) |\mathring{A}|^2|H|^4\\
&+C(N_1,N_2)|H|^4+(C(N_1,N_2)+C)|H|^2+C(N_1,N_2)\\
&-\frac{2}{n}\eta|H|^6 +C\eta,
\end{split}
\end{equation}
where $C$ is as in (\ref{18}) and (\ref{19}). Now we first choose
$N_2$ large enough such that the gradient term on the first line of
(\ref{24}) is nonpositive. Then we choose $N_1$ large enough such
 that the gradient term on the second line of (\ref{24}) can be
 absorbed. The remained terms can be estimated by using Theorem
 \ref{thm-pinching-|A0|^2} and Young's inequality, which gives
\begin{equation}\label{25}
\begin{split}
\frac{\partial}{\partial t} f\leq& \Delta f+C(N_1,N_2,\eta),
\end{split}
\end{equation}
where $C(N_1,N_2,\eta)$ is some positive constant depending on
${N_1,N_2,\eta}$ and others but independent of $t$. Since the
maximal existence time of the mean curvature flow is finite, we
conclude that $f\leq C_\eta$. Then the theorem follows from the
definition of $f$.

\end{proof}

\section{Contraction to a round point}

In this section we show that as time tends to $T$, the
submanifold will shrink to a single point. If we dilate the metric of the ambient
space by a factor, which is a function of $t$, then the submanifold
will maintain the volume. After a reparameterization of time,
the dilated submanifold converges to a totally umbilical sphere in
the Euclidean space as the reparameterized time tends to infinity.

We need the following lemma.
\begin{lemma}[\cite{XG}]\label{Xu-Gu-ineq}Let $M^n$ be an $n$-dimensional submanifold in an $(n+d)$-dimensional Riemannian
manifold $N^{n+d}$, and $\pi$ a tangent 2-plane on $T_xM$ at point
$x\in M$. Choose an orthonormal two-frame $\{e_1, e_2\}$ at $x$ such
that $\pi = {\rm span}\{e_1, e_2\}$. Then
\begin{equation*}
K(\pi) \geq \frac{1}{2} \Big( 2\bar{K}_{\min} +
\frac{|H|^2}{n-1}-|A|^2 \Big)
+\sum_{\alpha=n+1}^{n+d}\sum_{j>i,(i,j)\neq(1,2)}(h^\alpha_{ij})^2.
\end{equation*}
\end{lemma}

In Lemma \ref{Xu-Gu-ineq}, $K(\pi) $ is the sectional curvature of
$M$ for the 2-plane $\pi$, and $\bar{K}_{\min}$ is the minimum of
the sectional curvature of $N$ at point $x$. From our assumption and
$a_\varepsilon<\frac{1}{n-1}$,  we see that the sectional curvature
$K_M$ of the evolving submanifold satisfies $K_M\geq
\epsilon^2|H|^2>0$ for some constant $\epsilon$ provided we choose
$b_0=\max \{b_1, 2 {K}_1\}$. Note that $b_0=0$ if $K_1+K_2=0$. With
the same argument as in \cite{Andrews-Baker} we have
$\frac{|H|_{\max}}{|H|_{\min}}\rightarrow 1$ and ${\rm
diam}M\rightarrow 0$ as $t\rightarrow T$. Hence the submanifold
shrinks to a single point $P\in N$ along the mean curvature flow.

To see that the evolving submanifold becomes spherical, we dilate
the metric of the ambient space such that the submanifold with the
induced metric by the immersion has fixed volume along the flow. Let
$\psi$ be a function of $t$ satisfying
\begin{equation*}
\psi^{-1}\frac{d\psi}{dt}=\frac{1}{n}\frac{\int_{M_t}|H|^2d\mu_t}{\int_{M_t}d\mu_t}:=\frac{\hbar}{n}.
\end{equation*}

Let $h$ be the Riemannian metric on $N$. Now we dilate the metric
$h$ such that $(N, \psi(t)^2 h)$, $t\in [0,T)$ is a family of
Riemannian manifolds. Let $\tilde{g}(t)$ be the induced metric on
the submanifold $M$ from $(N, \psi(t)^2 h)$ by the immersion $F_t$.
We denote by $(\tilde{M},\tilde{g}(t))$ the dilated submanifold with
the isometric immersion $\tilde{F}_t$, where $\tilde{M}=M$ and
$\tilde{F}_t=F_t$.
 We also have the following relations.
 \begin{lemma}\label{dilation-equalities}
 \begin{equation*}\begin{split}
\tilde{A}&=  A, \\
\tilde{H}&=\psi^{-2}H,\\
|\tilde{A}|^2&=\psi^{-2}|A|^2,\\
|\tilde{H}|^2&=\psi^{-2}|H|^2,\\
d\tilde{\mu}_{\tilde{g}(t)}&=\psi^n d\mu_{g(t)}, \\
\tilde{\nabla}&=\nabla,\\
\tilde{\Delta}&=\psi^{-2}\Delta.
 \end{split}
 \end{equation*}
 \end{lemma}
\begin{proof}
Let $\{x^i\}$ be a local coordinate system on $M$.
 For the induced metric, we have the following
\begin{equation*}\begin{split}
\tilde{g}_{ij}&=\bigg\langle \frac{\partial\tilde{ F}}{\partial
x^i},\frac{\partial\tilde{ F}}{\partial x^j}
\bigg\rangle_{\psi^2h}\\
& =\bigg\langle \frac{\partial {
F}}{\partial x^i},\frac{\partial {
F}}{\partial x^j} \bigg\rangle_{\psi^2h}\\
&=\psi^2\bigg\langle \frac{\partial { F}}{\partial
x^i},\frac{\partial { F}}{\partial x^j}
\bigg\rangle_{h}\\
&=\psi^2g_{ij}.
\end{split}
\end{equation*}

For the second fundamental form, noting that
$\tilde{e}_\alpha=\psi^{-1} {e}_\alpha$, we have
\begin{equation*}
\begin{split}
\tilde{h}^\alpha_{ij}&=-\bigg\langle  \frac{\partial^2 \tilde{
F}}{\partial x^i
\partial x^j} , \tilde{e}_\alpha \bigg\rangle_{\psi^2 h}\\
&=-\psi^2 \bigg\langle  \frac{\partial^2  { F}}{\partial x^i
\partial x^j} , \psi^{-1} {e}_\alpha \bigg\rangle_{ h}\\
&=\psi h^\alpha_{ij}. \end{split}
\end{equation*}
Hence \begin{equation*}\tilde{h}_{ij}=\sum_\alpha
\tilde{h}_{ij}^\alpha \tilde{e}_\alpha=\sum_\alpha \psi
{h}_{ij}^\alpha \psi^{-1}{e}_\alpha=h_{ij},\end{equation*}
 which means
$\tilde{A}=A$.

The mean curvature is the trace of the second fundamental form, so
\begin{equation*} \tilde{H}=\tilde{g}^{ij}\tilde{h}_{ij}=\psi^{-2}g^{ij}
{h}_{ij}=\psi^{-2}H.
\end{equation*}

For the squared norm of the second fundamental form and the mean
curvature, we have
\begin{equation*}\begin{split}
|\tilde{A}|^2&= \tilde{g}^{ik}\tilde{g}^{kl}\tilde{h}_{ij}^\alpha
\tilde{h}_{ij}^\beta \langle \tilde{e}_\alpha,\tilde{e}_\beta
\rangle_{\psi^2 h}\\
&=\psi^{-2}{g}^{ik} {g}^{kl} {h}_{ij}^\alpha
 {h}_{ij}^\beta \langle e_\alpha,e_\beta \rangle_{
 h}\\
 &=\psi^{-2}|A|^2,
 \end{split}
\end{equation*}
\begin{equation*}
\begin{split}
|\tilde{H}|^2&=\langle \tilde{H},\tilde{H} \rangle_{\psi^2
h}\\
&=\psi^2\langle \psi^{-2}H,\psi^{-2}H \rangle_{h}\\
&=\psi^{-2}|H|^2.
\end{split}
\end{equation*}

The relation of the volume forms follow from $\tilde{g}=\psi^2 g$.
The Christoffel symbols are scale invariant and thus the connection
is also scale invariant. Finally, from the definition of the
Laplacian and the scalar invariance of the connection, we see that
$\tilde{\Delta}=\psi^{-2}\Delta.$
\end{proof}

 It is easy to check that
\begin{equation*}\frac{d}{d t} \int_{\tilde{M}}d\tilde{\mu}_{\tilde{g}(t)}=0,
\end{equation*}
which means that the volume of the dilated submanifold is fixed as
$t$ tends to $T$.

Now we define the rescaled time variable $\tilde{t}$ by
\begin{equation*}
\tilde{t}(t)=\int_{0}^t\psi^2(\tau)d\tau.
\end{equation*}
So $\frac{d \tilde{t}}{dt}=\psi^2.$

For $0\leq \tilde{t}<\tilde{T}=\tilde{t}(T)$, we first has the
following estimates.
\begin{equation}\label{26}
|\tilde{A}|^2\leq C_n|\tilde{H}|^2,
\end{equation}
\begin{equation}\label{27}
\frac{|\tilde{H}|_{\min}}{|\tilde{H}|_{\max}}\rightarrow 1 \ as\
\tilde{t}\rightarrow \tilde{T},
\end{equation}
\begin{equation}\label{28}
\tilde{K}_{\min}\geq \epsilon^2 |\tilde{H}|^2.
\end{equation}

For the induced metric we have the following evolution equation.
\begin{equation}
\begin{split}
\frac{\partial}{\partial
\tilde{t}}\tilde{g}_{ij}&=\psi^{-2}\frac{\partial}{\partial
t}\tilde{g}_{ij}\\
&=\psi^{-2}\Big(\frac{\partial}{\partial t}\psi^2 g_{ij}+\psi^2
\frac{\partial}{\partial t}g_{ij} \Big)\\
&=\psi^{-2}\Big( 2\psi^2 \frac{\hbar}{n}g_{ij}-\psi^22H\cdot h_{ij}
\Big)\\
&=-2H\cdot h_{ij}+\frac{2}{n}\hbar g_{ij}\\
&=-2\tilde{H} \tilde{\cdot} \tilde{h}_{ij}+\frac{2}{n}\tilde{\hbar}
\tilde{g}_{ij}.
\end{split}
\end{equation}
Here $\tilde{\cdot}$ denotes the inner product with respect to the
metric $\psi^2 h$. The volume form $d\tilde{\mu}_{\tilde{t}}$
satisfies the following equation.
\begin{equation}\label{29}
\frac{\partial}{\partial
\tilde{t}}d\tilde{\mu}_{\tilde{t}}=(\tilde{\hbar}-|\tilde{H}|^2)d\tilde{\mu}_{\tilde{t}}.
\end{equation}

\begin{prop}\label{mean-curvature-bound}
$0<C_{\min}\leq |\tilde{H}|_{\min} \leq |\tilde{H}|_{\max}\leq
C_{\max}<\infty$ holds for  $0\leq \tilde{t}<\tilde{T}$.
\end{prop}

\begin{proof}
From (\ref{28}), the sectional curvature of $\tilde{M}_{\tilde{t}}$
is nonnegative. The Bishop-Gromov volume comparison theorem implies
that ${\rm Vol} (\tilde{M}_{\tilde{t}})\leq C \tilde{d}^n$, where
$\tilde{d}$ is the diameter of $\tilde{M}_{\tilde{t}}$. From the
Bonnet theorem, we also have $\tilde{d}\leq
\frac{\pi}{\sqrt{|\tilde{H}|_{\min}}}$. On the other hand, since
${\rm Vol}(\tilde{M}_{\tilde{t}})={\rm Vol}(\tilde{M}_{{0}})$, we
have $|\tilde{H}|_{\min}\leq C$, and then (\ref{27}) implies
$|\tilde{H}|_{\max}\leq C_{\max}$.

If we can show that $|\tilde{H}|_{\max}\geq C>0$, then  (\ref{27})
implies $|\tilde{H}|_{\min}\geq C_{\min}>0$ for all $\tilde{t}\in
[0,\tilde{T})$. Suppose  $|\tilde{H}|_{\max}\rightarrow 0$ as
$\tilde{t}\rightarrow\tilde{T} $, (\ref{26}) implies that
$|\tilde{A}|_{\max}^2\rightarrow 0$ as
$\tilde{t}\rightarrow\tilde{T}$.   Since $|H|^2$ satisfies
\begin{equation*}
\frac{\partial}{\partial t} |H|^2\leq \Delta  |H|^2
+C|H|^2_{\max}|H|^2,
\end{equation*}
we can follow the argument in \cite{Ha} to show that $\int_0^T
|H|^2_{\max}dt=\infty$. On the other hand, since
$\frac{|H|_{\max}}{|H|_{\min}}\rightarrow 1$  as $t\rightarrow T$,
then for a   $\varsigma>0$£¬ there is a positive constant $\delta
>0$ such that $\frac{|H|_{\max}}{|H|_{\min}}\leq 1+\varsigma$ for
all $t\in [\delta,T)$.  So
\begin{equation*}
\begin{split}\infty&=
\frac{1}{(1+\varsigma)^2}\int_0^T|H|^2_{\max}dt\\
&=\frac{1}{(1+\varsigma)^2}\int_0^\varsigma|H|^2_{\max}dt+\frac{1}{(1+\varsigma)^2}\int_\varsigma^T|H|^2_{\max}dt\\
&\leq\frac{1}{(1+\varsigma)^2}\int_0^\varsigma|H|^2_{\max}dt+
\int_\varsigma^T|H|^2_{\min}dt\\
&\leq\frac{1}{(1+\varsigma)^2}\int_0^\varsigma|H|^2_{\max}dt+
\int_0^T|H|^2_{\min}dt.
\end{split}\end{equation*}
This implies $\int_0^T|H|^2_{\min}dt=\infty$ since
$\int_0^\varsigma|H|^2_{\max}dt<\infty$. We also have $\frac{d
\tilde{t}}{dt}=\psi^2$ and $|\tilde{H}|^2=\psi^{-2}|H|^2$, hence
\begin{equation*}
\int_0^{\tilde{T}}\tilde{\hbar}(\tilde{t})d\tilde{t}=\int_0^{{T}}{\hbar}({t})d{t}\geq
\int_0^{{T}}|H|^2_{\min}(t)d{t}=\infty.
\end{equation*}
However, we have $\tilde{\hbar}\leq |\tilde{H}|^2_{\max}\leq
C_{\max}^2$. Therefore $\tilde{T}=\infty$. By the definition of the
rescaling of the time variable,
$\tilde{T}=\tilde{t}(T)=\int_0^T\psi^2(\tau)d\tau$. Since
$T<\infty$, we have $\psi(t)\rightarrow \infty$ as $t\rightarrow T$.
This implies that $(N, \psi^2 h,P)$ converges to the Euclidean space
as $\tilde{t}\rightarrow \infty$. Since we have
$|\tilde{A}|_{\max}^2\rightarrow 0$ as
$\tilde{t}\rightarrow\tilde{T}$, the family of immersions
$\tilde{F}:\, \tilde{M}\rightarrow (N, \psi^2 h)$ for $\tilde{t}\in
[0,\infty)$, will converge to the isometric immersion of an
$n$-dimensional Euclidean space into an $(n+d)$-dimensional
Euclidean space as $\tilde{t}\rightarrow \infty$. This is a
contradiction since the volume of $\tilde{M}_{\tilde{t}}$ is
unchanged along the flow. This competes the proof of the
proposition.
\end{proof}

Now we look at the tracefree second fundamental form
$\mathring{\tilde{A}}$ of the immersion $\tilde{F}$. From Theorem
\ref{thm-pinching-|A0|^2}, we have the following inequality
\begin{equation*}
|\mathring{\tilde{A}}|^2\leq C_0 \psi^{-\delta}
|\tilde{H}|^{2-\delta},
\end{equation*}
which holds for all $\tilde{t}\in [0,\tilde{T})$. As in the proof of
Proposition \ref{mean-curvature-bound}, we have $\tilde{T}=\infty$
and $\psi(\tilde{t})\rightarrow \infty$ as
$\tilde{t}\rightarrow\infty$. Hence there holds
\begin{equation*}|\mathring{\tilde{A}}|^2\rightarrow 0,
\ \ \frac{|\tilde{A}|^2}{|\tilde{H}|^2}-\frac{1}{n} \rightarrow 0\ \
\textrm{as} \ \ \tilde{t}\rightarrow \infty.\end{equation*} Since
$(N, \psi^2(\tilde{t}) h,P)\rightarrow (\mathbb{R}^{n+d},
\delta_{AB},0)$ in $C^\infty_{loc}$-topology, we see that
$\tilde{F}(\tilde{t})$ tends to a totally umbilical immersion from a
standard sphere with the same volume as $M_0$ into the Euclidean
space at least in the $C^0$-topology.

To see that the convergence is in $C^\infty$-topology, we only need
to show that all the covariant derivatives of the second fundamental
form are uniformly bounded. The second fundamental form $\tilde{A}$
is uniformly bounded by (\ref{26}) since the mean curvature is
uniformly bounded.

Along the mean curvature flow, $M_t$ will stay in a compact region
of $N$, say, $M_t\subset B_{h}(P,r)$ for $t\in [0,T)$ and a suitable
$r>0$. After the dilation, $\tilde{M}_{\tilde{t}}\subset
B_{h}(P,r)$, where $B_{h}(P,r)$ is in fact equal to
$B_{\psi^2(\tilde{t})h}(P,\psi(\tilde{t})r)$ as a set. By the
definition of the dilation, the Riemannian curvature and its
covariant derivations of the ambient space, when restricted to
$\tilde{M}_{\tilde{t}}$, are uniformly bounded by constants
independent of $\tilde{t}$.

\begin{lemma}\label{dilation-evo}
Let $P$ and $Q$ be two quantities depending on $g$ and $A$, and $P$
satisfies the evolution equation $\frac{\partial P}{\partial
t}=\Delta P+Q$ along the mean curvature flow. If $P$ has degree
$\alpha$, that is, $\tilde{P}=\psi^\alpha P$, then $Q$ has degree
$\alpha-2$ and after dilation,  $\tilde{P}$ satisfies
\begin{equation*}
\frac{\partial \tilde{P}}{\partial
\tilde{t}}=\tilde{\Delta}\tilde{P}+\tilde{Q}+\frac{\alpha}{n}\tilde{\hbar}\tilde{P}.
\end{equation*}
\end{lemma}
\begin{proof}
Using the similar computation as in the proof of Lemma 9.1 in
\cite{H1}, we can prove the lemma.
\end{proof}

We also need the following interpolation inequalities.
\begin{lemma}\label{interpolation}
Let $T$ be any tensor and $m$ an integer. There is a constant
$C(n,m)$ independent of the metric and the connection such that\\
\begin{equation*}\int_{M} |\nabla^iT|^{\frac{2m}{i}}d\mu\leq C\cdot \max_{M}|T|^{\frac{2m}{i}-2}\int_{M} |\nabla^mT|^2d\mu\end{equation*}
holds for any $1\leq i\leq m-1$, and
\begin{equation*}
\int_{M}|\nabla^iT|^2d\mu\leq C\bigg(\int_{M}|\nabla^mT|^2d\mu
\bigg)^{\frac{i}{m}}\bigg(\int_{M}|T|^2d\mu \bigg)^{1-\frac{i}{m}}
\end{equation*} holds for any $0\leq i\leq m$.
\end{lemma}
\begin{proof}
These interpolation inequalities were proved in \cite{Ha}.
\end{proof}

From Lemma \ref{dilation-evo}, we have the following inequality
after dilation.

\begin{equation}\label{evo-derivatives-dilation}
\begin{split}
\frac{\partial}{\partial \tilde{t}}|\tilde{\nabla}^m
\tilde{A}|^2\leq &\tilde{\Delta}|\tilde{\nabla}^m
\tilde{A}|^2-2|\tilde{\nabla} ^{m+1}\tilde{A}|^2\\
&+C(n,m)\Big\{ \sum_{i+j+k=m} |\tilde{\nabla}^i
\tilde{A}||\tilde{\nabla}^j \tilde{A}||\tilde{\nabla}^k
\tilde{A}||\tilde{\nabla}^m \tilde{A}|\\
& +\bar{C}_m \sum_{i\leq m} |\tilde{\nabla}^i
\tilde{A}||\tilde{\nabla}^m \tilde{A}|
+\bar{\bar{C}}_m|\tilde{\nabla}^m \tilde{A}| \Big\}\\
&+\tilde{C}(n,m)\tilde{\hbar}|\tilde{\nabla}^m \tilde{A}|^2.
\end{split}
\end{equation}

To prove that for any integer $m\geq0$, $|\tilde{\nabla}^m
\tilde{A}|^2$ is uniformly bounded on $[0,\infty)$, we argue by
induction on $m$. Obviously, for $m=0$, $|\tilde{A}|^2$ is uniformly
bounded. Suppose that $|\tilde{\nabla}^i \tilde{A}|^2$ is uniformly
bounded on $[0,\infty)$ for $i=1,2,\cdots, m-1$. Then from
(\ref{evo-derivatives-dilation}), we see that there holds
\begin{equation}\label{evo-derivatives-dilation11}
\begin{split}
\frac{\partial}{\partial \tilde{t}}|\tilde{\nabla}^m
\tilde{A}|^2\leq &\tilde{\Delta}|\tilde{\nabla}^m
\tilde{A}|^2-2|\tilde{\nabla} ^{m+1}\tilde{A}|^2+C|\tilde{\nabla}^m
\tilde{A}|^2+C
\end{split}
\end{equation}
for some positive constant $C$ independent of $\tilde{t}$. Here we
have used Young's inequality and the boundedness of $\tilde{\hbar}$
which is implied by that the mean curvature is uniformly bounded and
the area of the submanifold is fixed. From
(\ref{evo-derivatives-dilation11}) we have
\begin{equation}\label{integral-ineq-dilation}\begin{split}
\frac{d}{d \tilde{t}}\int_{\tilde{M}_{\tilde{t}}} |\tilde{\nabla}^m
\tilde{A}|^2 d{\tilde{\mu}}_{\tilde{t}}=&
\int_{\tilde{M}_{\tilde{t}}} \frac{\partial}{\partial
\tilde{t}}|\tilde{\nabla}^m \tilde{A}|^2 d{\tilde{\mu}}_{\tilde{t}}
+\int_{\tilde{M}_{\tilde{t}}} |\tilde{\nabla}^m \tilde{A}|^2
\frac{\partial}{\partial \tilde{t}}d{\tilde{\mu}}_{\tilde{t}}\\
\leq &-2\int_{\tilde{M}_{\tilde{t}}} |\tilde{\nabla}^{m+1}
\tilde{A}|^2
d{\tilde{\mu}}_{\tilde{t}}+C\int_{\tilde{M}_{\tilde{t}}}
|\tilde{\nabla}^m \tilde{A}|^2 d{\tilde{\mu}}_{\tilde{t}}\\
&+\int_{\tilde{M}_{\tilde{t}}} |\tilde{\nabla}^m
\tilde{A}|^2(\tilde{\hbar}-|\tilde{H}|^2)
d{\tilde{\mu}}_{\tilde{t}}+C\\
\leq &-2\int_{\tilde{M}_{\tilde{t}}} |\tilde{\nabla}^{m+1}
\tilde{A}|^2
d{\tilde{\mu}}_{\tilde{t}}+C\int_{\tilde{M}_{\tilde{t}}}
|\tilde{\nabla}^m \tilde{A}|^2 d{\tilde{\mu}}_{\tilde{t}}+C.
\end{split}\end{equation}

By the second interpolation inequality in Lemma \ref{interpolation},
we have
\begin{equation}\label{30}
\begin{split}
\int_{\tilde{M}_{\tilde{t}}} |\tilde{\nabla}^m \tilde{A}|^2
d{\tilde{\mu}}_{\tilde{t}}\leq& \bar{C}\bigg(
\int_{\tilde{M}_{\tilde{t}}} |\tilde{\nabla}^{m+1} \tilde{A}|^2
d{\tilde{\mu}}_{\tilde{t}}\bigg)^{\frac{m}{m+1}}\bigg(\int_{\tilde{M}_{\tilde{t}}}
|\tilde{A}|^2 d{\tilde{\mu}}_{\tilde{t}} \bigg)^{\frac{1}{m+1}}\\
\leq&\bar{C}\eta \int_{\tilde{M}_{\tilde{t}}} |\tilde{\nabla}^{m+1}
\tilde{A}|^2 d{\tilde{\mu}}_{\tilde{t}}+\bar{C}\eta^{-m}
\int_{\tilde{M}_{\tilde{t}}} |\tilde{A}|^2
d{\tilde{\mu}}_{\tilde{t}}\\
\leq& \bar{C}\eta\int_{\tilde{M}_{\tilde{t}}} |\tilde{\nabla}^{m+1}
\tilde{A}|^2 d{\tilde{\mu}}_{\tilde{t}} +\bar{\bar{C}}\eta^{-m}.
\end{split}\end{equation}

Combining (\ref{integral-ineq-dilation}) and (\ref{30}), we have
\begin{equation}\label{integral-ineq-dilation1}\begin{split}
\frac{d}{d \tilde{t}}\int_{\tilde{M}_{\tilde{t}}} |\tilde{\nabla}^m
\tilde{A}|^2 d{\tilde{\mu}}_{\tilde{t}} \leq
&-2\int_{\tilde{M}_{\tilde{t}}} |\tilde{\nabla}^{m+1} \tilde{A}|^2
d{\tilde{\mu}}_{\tilde{t}}-C\int_{\tilde{M}_{\tilde{t}}}
|\tilde{\nabla}^m \tilde{A}|^2 d{\tilde{\mu}}_{\tilde{t}}+C\\
&+2C\bar{C}\eta\int_{\tilde{M}_{\tilde{t}}} |\tilde{\nabla}^{m+1}
\tilde{A}|^2 d{\tilde{\mu}}_{\tilde{t}} +2C\bar{\bar{C}}\eta^{-m}\\
=&-C\int_{\tilde{M}_{\tilde{t}}} |\tilde{\nabla}^m \tilde{A}|^2
d{\tilde{\mu}}_{\tilde{t}}+\tilde{C}.
\end{split}\end{equation}
Here we have chosen $\eta=(C\bar{C})^{-1}$ and
$\tilde{C}=C+2C\bar{\bar{C}}(C\bar{C})^m$. Put
$f(\tilde{t})=\int_{\tilde{M}_{\tilde{t}}} |\tilde{\nabla}^m
\tilde{A}|^2 d{\tilde{\mu}}_{\tilde{t}}$, then
\begin{equation*}\begin{split}
\frac{d}{d\tilde{t}}\Big(f-\frac{\tilde{C}}{C}\Big)=\frac{d}{d\tilde{t}}f
\leq -Cf+\tilde{C} =-C\Big(f-\frac{\tilde{C}}{C}\Big).
\end{split}\end{equation*}

If $f-\frac{\tilde{C}}{C}\leq 0$ at $\tilde{t}=0$, then it holds for
any $t\in [0,\infty)$, which implies $\int_{\tilde{M}_{\tilde{t}}}
|\tilde{\nabla}^m \tilde{A}|^2 d{\tilde{\mu}}_{\tilde{t}}\leq C_m$
for some positive constant $C_m$ depending on $m$ but independent of
$\tilde{t}$.

If $f-\frac{\tilde{C}}{C}> 0$ at $\tilde{t}=0$, then
$f-\frac{\tilde{C}}{C}\leq
(f(0)-\frac{\tilde{C}}{C})e^{-C\tilde{t}}$. This also implies that
$\int_{\tilde{M}_{\tilde{t}}} |\tilde{\nabla}^m \tilde{A}|^2
d{\tilde{\mu}}_{\tilde{t}}\leq C_m$ for some positive constant $C_m$
depending on $m$ but independent of $\tilde{t}$.

From the first interpolation inequality in Lemma
\ref{interpolation}, we see that
\begin{equation*}
\int_{\tilde{M}_{\tilde{t}}} |\tilde{\nabla}^m \tilde{A}|^p
d{\tilde{\mu}}_{\tilde{t}}\leq C_{m,p}
\end{equation*}
holds for large $p$. Combining this with the Sobolev inequality in
\cite{HS}, we may carry out a Stampacchia iteration process just as
in \cite{H1} and \cite{H2} to get that $|\tilde{\nabla}^m
\tilde{A}|^2\leq \tilde{C}_m$ for a constant $\tilde{C}_m<\infty$
depending on $m$.

Hence we have proved that the convergence is in fact in
$C^\infty$-topology.


\begin{thebibliography}{10}

\bibitem{Andrews-Baker}B.  Andrews and  C. Baker: \textit{Mean curvature flow of pinched submanifolds to
spheres}, J. Differential Geom. \textbf{85}(2010), 357-395.


\bibitem{Baker}C. Baker: \textit{The mean curvature flow of submanifolds
of high codimension}, arXiv:1104.4409v1.


\bibitem{B} K. Brakke: The motion of a surface by its mean
curvature, Princeton, New Jersey: Princeton University Press, 1978.

\bibitem{Goldberg} S. Goldberg: Curvature and Homology, Academic Press, London, 1962.

\bibitem{Gu-Xu}J. R. Gu and H. W. Xu: \textit{The sphere theorems for manifolds with positive scalar
curvature}, arXiv:1102.2424v1.

\bibitem{Ha} R. Hamilton: \textit{Three-manifolds with positive Ricci curvature}, J. Differential Geom. \textbf{17}(1982),
255-306.

\bibitem{Ha2} R. Hamilton: \textit{Four-manifolds with positive curvature operator}, J. Differential Geom. \textbf{24}(1986),
153-179.

\bibitem{HS} D. Hoffman and J. Spruck: \textit{Sobolev and isoperimetric inequalities for
Riemannian submanifolds}, Comm. Pure Appl. Math. \textbf{27}(1974),
715-727.

\bibitem{H1} G. Huisken: \textit{Flow by mean curvature of convex surfaces into
spheres}, J. Differential Geom. \textbf{20}(1984), 237-266.

\bibitem{H2} G. Huisken: \textit{Contracting convex hypersurfaces in Riemannian manifolds by their mean curvature},
Invent. Math. \textbf{84}(1986), 463-480.

\bibitem{H3} G. Huisken:  \textit{Deforming hypersurfaces of the sphere by their mean curvature}, Math. Z. \textbf{195}(1987),
205-219.

\bibitem{H4} G. Huisken: \textit{Asymptotic behavior for singularities
of the mean curvature flow}, J. Differential Geom. \textbf{31}(1990), 285-299.

\bibitem{LXYZ1}K. F. Liu, H. W. Xu, F. Ye and E. T. Zhao: \textit{The extension and convergence of mean curvature flow in higher
codimension},  arXiv:1104.0971v1.


\bibitem{LXYZ2} K. F. Liu, H. W. Xu, F. Ye and E. T. Zhao:  \textit{Mean curvature flow
of higher codimension in hyperbolic spaces}, arXiv:1105.5686v1.


\bibitem{LXZ}K. F. Liu, H. W. Xu and E. T. Zhao: \textit{Deforming submanifolds of arbitrary codimension in a sphere}, preprint, 2011.


\bibitem{Mu}W. W. Mullins: \textit{Two-dimensional motion of idealized
grain boundaries}, J. Appl. Phys. {\bf27}(1956), 900-904.


\bibitem{Shiohama-Xu-97} K. Shiohama and H. W. Xu:  \textit{The topological
sphere theorem for complete submanifolds}, Compositio Math.
{\bf107}(1997), 221-232.

\bibitem{Shiohama-Xu-98}K. Shiohama and H. W. Xu: \textit{A general rigidity theorem for complete
submanifolds}, Nagoya Math. J. \textbf{150}(1998), 105-134.

\bibitem{Sm} K. Smoczyk: \textit{Longtime existence of the Lagrangian mean
curvature flow}, Calc. Var. \textbf{20}(2004), 25-46.

\bibitem{SW}K. Smoczyk and M. T. Wang: \textit{Mean curvature flows for Lagrangian
submanifolds with convex potentials}, J. Differential. Geom.
\textbf{62}(2002), 243-257.

\bibitem{WaM1} M. T. Wang: \textit{Mean curvature flow of surfaces in Einstein
four-manifolds}, J. Differential. Geom. \textbf{57}(2001), 301-338.

\bibitem{WaM4} M. T. Wang: \textit{Deforming area preserving diffeomorphism of surfaces
by mean curvature flow}, Math. Res. Lett. \textbf{8}(2001), 651-661.

\bibitem{WaM2} M. T. Wang: \textit{Long-time existence and convergence of graphic mean curvature flow
in arbitrary codimension}, Invent. Math. \textbf{148}(2002),
525-543.

\bibitem{WaM5} M. T. Wang: \textit{Gauss maps of the mean curvature flow}, Math. Res. Lett. \textbf{10}(2003), 287-299.

\bibitem{WaM6} M. T. Wang: \textit{The mean curvature flow smoothes Lipschitz submanifolds}, Comm. Anal. Geom. \textbf{12}(2004), 581-599.

\bibitem{WaM3} M. T. Wang: \textit{Lectures on mean curvature flows in higher codimensions},
Handbook of geometric analysis. No. 1, 525-543, Adv. Lect. Math.
(ALM) 7, International Press, Somerville, MA, 2008.

\bibitem{XG} H. W. Xu and J. R. Gu: \textit{An optimal differentiable sphere theorem for complete manifolds},
 Math. Res. Lett. \textbf{17}(2010), 1111-1124.

\bibitem{XYZ} H. W. Xu, F. Ye and E. T. Zhao: \textit{The extension for mean curvature
flow with finite integral curvature in Riemannian manifolds},
Science China Math. \textbf{54}(2011), 2195-2204.

\bibitem{Xu-Zhao} H. W. Xu and E. T. Zhao: \textit{Topological and differentiable sphere theorems for complete submanifolds},
Comm. Anal. Geom. \textbf{17}(2009),  565-585.


\bibitem{Zhu} X. P. Zhu: \textit{Lectures on mean curvature flows}, Studies in Advanced Mathematics 32, International Press,
Somerville, MA, 2002.






\end{thebibliography}
\end{document}